\documentclass[reqno,11pt]{amsart}
\usepackage[onehalfspacing]{setspace}
\usepackage{mathtools}
\usepackage{amsfonts}
\usepackage{amssymb}
\usepackage{amsthm}
\usepackage{enumerate}
\usepackage{xfrac}
\usepackage{color}
\usepackage{graphicx}
\usepackage{caption}
\usepackage{mathrsfs}

\usepackage[left=1in,right=1in,top=1in,bottom=1in]{geometry}

\newtheorem{theorem}{Theorem}[section]
\newtheorem{lemma}[theorem]{Lemma}
\newtheorem{proposition}[theorem]{Proposition}

\newtheorem{definition}[theorem]{Definition\rm}

\newtheorem{conjecture}{Conjecture}

\newcommand{\R}{\mathbb R}
\newcommand{\N}{\mathbb N}

\renewcommand{\S}{\mathbb S}
\newcommand{\Lip}{{\rm Lip}}

\begin{document}
\pagenumbering{arabic}
\title{$C^{1,\alpha}$ isometric embeddings of polar caps}
\author[De Lellis]{Camillo De Lellis}
\address{School of Mathematics, Institute for Advanced Study, 1 Einstein Dr., Princeton NJ 05840, USA\\ and Princeton University\\
and Universit\"at Z\"urich}
\email{camillo.delellis@math.uzh.ch}

\author[Inauen]{Dominik Inauen}
\address{Institut f\"ur Mathematik, Universit\"at Z\"urich, CH-8057 Z\"urich}
\email{dominik.inauen@math.uzh.ch}

\begin{abstract} We study isometric embeddings of $C^2$ Riemannian manifolds in the Euclidean space and we establish that the H\"older
space $C^{1,\frac{1}{2}}$ is critical in a suitable sense: in particular we prove that for $\alpha > \frac{1}{2}$ the Levi-Civita connection
of any isometric immersion is induced by the Euclidean connection, whereas for any $\alpha < \frac{1}{2}$ we construct $C^{1,\alpha}$ 
isometric embeddings of portions of the standard $2$-dimensional sphere for which such property fails.
\end{abstract}
\maketitle

\section{Introduction}

In this paper we investigate the flexibility and rigidity of $C^{1,\alpha}$ isometric embeddings of Riemannian manifolds in  Euclidean spaces. Following standard notation,
if $(\Sigma, g)$ is a $C^1$ Riemannian manifold and $v: \Sigma\to \mathbb R^N$ is a $C^1$ immersion, we denote by $e$ the standard Euclidean metric on $\mathbb R^N$ and
by $v^{\sharp}e$ its pull-back on $\Sigma$: $v$ is isometric if and only if 
$v^\sharp e = g$. 

The outcome of our investigations is that, when we consider $C^{1,\alpha}$ isometric embeddings, the H\"older exponent $\alpha_0 = \frac{1}{2}$ is a threshold in the following sense. When $\alpha > \frac{1}{2}$ and $v$ is a $C^{1,\alpha}$ isometric immersion of a $C^2$ Riemannian manifold $(\Sigma, g)$, the Levi-Civita connection of $(\Sigma, g)$ agrees with the connection induced by the ambient (Euclidean) one. Instead, for any $\alpha < \frac{1}{2}$ we can produce isometric immersions for which the Levi-Civita connection induced by the ambient differs from the one compatible with $g$. While we prove the first statement in full generality, cf. Proposition
\ref{p:p=p}, concerning the second statement we defer the most general versions to a forthcoming work. 
In this note we focus instead on a particular case which, in our opinion, provides the cleanest illustration of the criticality of the exponent $\alpha = \frac{1}{2}$ in Theorem \ref{t:criticality} below.

Consider the standard $2$-dimensional sphere as the subset $\mathbb S^2 := \{x : x_1^2+x_2^2 +x_3^2 =1\}\subset \mathbb R^3$ and for $a\in ]-1,1[$ denote by $(\Sigma_a, \sigma)$ the Riemannian manifold (with boundary) given by 
\begin{equation}\label{e:polar_cap}
\Sigma_a = \mathbb S^2 \cap \{x_3 \geq a\} = \{x\in \mathbb R^3 : x_1^2+x_2^2 + x_3^3 =1 \;\mbox{and}\; x_3\geq a\}
\end{equation} 
equipped with the standard metric $\sigma$ as submanifold of $\mathbb R^3$. 
\begin{definition}
We denote by $\mathscr{I}^\alpha_k (\Sigma_a)$ the space of isometric immersions $v: \Sigma_a \to \mathbb R^{2+k}$ of class $C^{1,\alpha}$ with the property that $v (x_1, x_2, a) = (x_1, x_2, 0, \ldots, 0)$ for all $(x_1, x_2, a)\in \partial \Sigma_a$. Moreover we denote by $\gamma_a$ the circle
$v (\partial \Sigma_a)$. 
\end{definition}
In what follows $\langle x,y\rangle$ denotes the scalar product of vectors $x,y\in \mathbb R^m$. 
\begin{theorem}\label{t:criticality} 
Let $X$ be the interior unit normal to $\partial \Sigma_a$ in $\Sigma_a$ and $Z: \gamma_a \to \mathbb R^{2+k}$ the unit vector field $Z (x_1, x_2, 0, \ldots , 0) = - (1-a^2)^{-\sfrac{1}{2}} (x_1, x_2, 0, \ldots , 0)$. 
For any element $v\in \mathscr{I}^\alpha_k (\Sigma_a)$ let $Y: \gamma_a \to \mathbb R^{2+k}$ be the vector field $v_* X$.
Then the following holds
\begin{itemize}
\item[(a)] If $\alpha > \frac{1}{2}$, $-1<a<1$, $k\geq 1$ and $v\in \mathscr{I}^\alpha_k (\Sigma_a)$, then $\langle Y, Z \rangle = a$.
\item[(b)] For any $\alpha < \frac{1}{2}$, $0<a<1$ and $k\geq 12$ there is $v\in \mathscr{I}^\alpha_k (\Sigma_a)$ such that $\langle Y, Z \rangle > a$. 
\end{itemize}
\end{theorem}

Our theorem is thus related to a question of Gromov on the criticality of the exponent $\frac{1}{2}$, cf. \cite[Section 3.5, Quest. C]{Gromov}, because the proof of part (b) follows a suitable modification of the celebrated Nash-Kuiper construction, cf. \cite{Nash54,Kuiper} and (a) is thus an obstruction to the implementation of such methods, at least in our context where a boundary condition is imposed.
Note indeed that without such restriction K\"allen in \cite{Kallen} is able to reach the threshold $C^{1,1}$: our theorem implies thus that the Nash-Kuiper construction and K\"allen's iteration differ in a rather nontrivial way.

Moreover, although in a weak sense, Theorem \ref{t:criticality} can be thought as an analog of the celebrated conjecture of Onsager on the energy conservation for nonsmooth solutions of the threedimensional incompressible Euler equations, cf. \cite{On1949,Eyink,CoETi1994,DS1,ChCoFrSh2008,DS-Bull1,DS-Inv,DS-JEMS,DS-Bull2,Isett,Bu2014,BDIS,BuDLeSz2016,DaSz2016,Isett2016,BDSV}. 

Indeed, we expect much stronger manifestations of the criticality of the exponent $\frac{1}{2}$ to hold for isometric embeddings.
First of all, we do not expect the codimension 12 for part (b) in Theorem \ref{t:criticality} to have any geometric meaning, but we conjecture that the same holds in any codimension:

\begin{conjecture}\label{c:criticality}
For any $\alpha < \frac{1}{2}$ and any $0<a<1$ there is $v\in \mathscr{I}^\alpha_1 (\Sigma_a)$ such that $\langle Y , Z \rangle > a$.
\end{conjecture}

It is possible to use the same ideas of this paper to show that indeed conclusion (b) of Theorem \ref{t:criticality} holds for every $\alpha < \alpha_0 (k)$, where $\alpha_0 (k)$ is an explicitely computable number. For $k=1$ such threshold is $\frac{1}{5}$ and this can be shown quickly using some of the results in \cite{Laszlo2}. In fact while we were completing our work we learned that the authors in \cite{Laszlo2} were dealing with
Nash-Kuiper constructions of $C^{1,\alpha}$ isometric embeddings of Riemannian manifolds which are prescribed at the boundary, although with a different purpose. In the $C^1$ case, such variant of the classical Nash-Kuiper construction was first given in \cite{Wasem}.

Concerning part (a) of Theorem \ref{t:criticality}, in the case of codimension $1$ a much stronger conclusion holds if $\alpha > \frac{2}{3}$: in that case any $v\in \mathscr{I}^\alpha_1 (\Sigma_a)$ must be the standard isometric embedding, namely $v (\Sigma_a) = \Sigma_a$, up to translations and rotations. This follows from classical works on the Monge-Amp\`ere equation after showing that $v (\Sigma_a)$ is (locally) convex. The latter property was first proved by Borisov in the fifties for isometric immersions of positively curved surfaces in a series of papers, cf. \cite{Borisov58-1,Borisov58-2,Borisov58-3,Borisov58-4,BorisovRigidity1,BorisovRigidity2}. A much shorter argument has been given more recently in \cite{CDS10}. Motivated by Borisov's result, the following conjecture on the isometric embeddings of positively curved $2$-dimensional surfaces in the euclidean threedimensional space seems quite natural and would provide
a much stronger version of the criticality of the H\"older exponent $\frac{1}{2}$.

\begin{conjecture}\label{c:criticality2}
Let $\Sigma$ be a $2$-dimensional compact Riemannian manifold (possibly with boundary) with positive Gauss curvature. Then:
\begin{itemize}
\item[(a)] For any $\alpha > \frac{1}{2}$ the image of any $C^{1,\alpha}$ isometric embedding $v$ of $\Sigma$ in $\mathbb R^3$ is locally convex (namely, for any $p\in \Sigma$ there is a neighborhood $U$ such that $v (U)$ is convex).
\item[(b)] For any $\alpha < \frac{1}{2}$ there is a $C^{1,\alpha}$ isometric embedding $v$ of $\Sigma$ in $\mathbb R^3$ which is not locally convex and in fact any short embedding can be uniformly approximated with $C^{1,\alpha}$ isometric embeddings. 
\end{itemize}
\end{conjecture}

The best result concerning part (b) of the Conjecture is contained in \cite{DIS15}, where the statement is shown for any $\alpha < \frac{1}{5}$ and when $\Sigma$ is topologically a disk. 

We finally remark that when $\Sigma$ is connected and has no boundary, namely it is topologically a $2$-dimensional sphere, the above conjecture would have the rather elegant outcome that $C^{1,\alpha}$ isometric embeddings in $\mathbb R^3$ are unique up to isometries of the ambient space for $\alpha > \frac{1}{2}$, whereas they are highly nonunique for $\alpha < \frac{1}{2}$. 

\section{Rigidity: Proof of Theorem \ref{t:criticality} (a)}

\subsection{Preliminaries}

We start by recalling some well known facts in the theory of distributions. Given a closed interval $[a,b]$ we will denote by
$C^{1,\alpha}_0 ([a,b])$ the Banach space which is the closure of $C^{1,\alpha}_c (]a,b[)$ in $C^{1,\alpha} ([a,b])$. Thus $C^{1,\alpha}_0 ([a,b])$ is to the subspace of $C^{1,\alpha}$ functions $\varphi$ for which $\varphi (a) = \varphi'(a)=
\varphi(b) = \varphi' (b) = 0$.
If $h$ is a continuous function, we then regard $h$ as an element of the dual space $(C^{1,\alpha}_0 ([a,b]))^*$ after identifying it with the linear map
\[
\varphi \mapsto \int h \varphi\, .
\]

\begin{lemma}\label{l:product}
Let $\alpha > \frac{1}{2}$ and $[a,b]\subset \mathbb R$ a closed interval. Consider
the bilinear map $\mathcal{B}: C^\alpha ([a,b])\times C^1 ([a,b]) \ni (f,g) \mapsto f g' \in C ([a,b])$. Then the map extends to a unique continuous bilinear map $\mathscr{B}: C^\alpha ([a,b]) \times C^\alpha ([a,b]) \to (C^{1,\alpha}_0 ([a,b]))^*$.
\end{lemma}

\begin{proof}
First of all, by translating and dilating we can assume that $[a,b] = [0, \pi]$. Secondly, every $C^\alpha$ function on $[0, \pi]$ can be extended to a $C^\alpha$ periodic function on $[-\pi, \pi]$ by reflection, whereas every $C^{1,\alpha}_0$ function on $[0, \pi]$ can be extended to a $C^{1,\alpha}$ periodic function on $[-\pi, \pi]$ by setting it equal to $0$ on $[-\pi, 0]$. 
The first extension maps $C^1$ functions into Lipschitz maps. If $f\in L^\infty (\mathbb S^1)$ and $g\in {\rm Lip} (\mathbb S^1)$, then 
$f g'$ is a well defined $L^\infty$ function on $[-\pi, \pi]$ by Rademacher's theorem, which in turn we can identify with an element of $(C^{1,\alpha} (\mathbb S^1))^*$ by integration. On the other hand for maps $\varphi\in C^{1,\alpha} (\mathbb S^1)$ which vanish on $[-\pi, 0]$ the integral $\int fg' \varphi$ takes place only on $[0, \pi]$. We have thus reduced to prove that the bilinear map 
\[
C^\alpha (\mathbb S^1) \times \Lip (\mathbb S^1) \ni (f,g) \mapsto fg'\in (C^{1,\alpha} (\mathbb S^1))^*
\] 
extends to a unique continuous bilinear operator $\mathscr{B} : C^\alpha (\mathbb S^1)\times C^\alpha (\mathbb S^1) \to\left ( C^{1,\alpha} (\mathbb S^1)\right )^*$. The uniqueness part is a consequence of the fact that for every $\psi \in C^\alpha (\mathbb S^1)$ we can find a sequence of Lipschitz maps $\{\psi_k\}$ which converge to $\psi$ in $C^\beta$ for every $\beta < \alpha$
and such that $\|\psi_k\|_{C^\alpha} \leq \|\psi\|_{C^\alpha}$. We thus just need to show the existence of a constant $C$ such that the estimate
\begin{equation}\label{e:estimate}
\left| \int f g' \varphi \right| \leq C \|f\|_{C^\alpha} \|g\|_{C^\alpha} \|\varphi\|_{C^{1,\alpha}}
\end{equation}
holds for every triple $f\in C^\alpha$, $g\in \Lip$ and $\varphi \in C^{1,\alpha} (\mathbb S^1)$. Taking the supremum over $\varphi\in C^{1,\alpha}$ with $\|\varphi\|_{C^{1,\alpha}} \leq 1$ the latter estimate gives indeed the bound
\begin{equation}
\|\mathcal{B} (f,g)\|_{(C^{1,\alpha})^*} \leq C \|f\|_{C^\alpha} \|g\|_{C^\alpha} \qquad \forall (f,g)\in \Lip \times C^\alpha\, .
\end{equation}
In turn this implies the local uniform continuity of the bilinear map $\mathcal{B}$, since we can simply use the bilinearity and the triangle inequality to estimate
\[
\|\mathcal{B} (f,g) - \mathcal{B} (h, k)\|_{(C^{1,\alpha})^*} \leq \|f\|_{C^\alpha} \|g-k\|_{C^\alpha} + \|f-h\|_{C^\alpha} \|k\|_{C^\alpha}\, .
\]
The existence and uniqueness of the continuous extension $\mathscr{B}$ is then an obvious fact.

\medskip

We next observe that, by a standard approximation procedure, it suffices to prove the estimate \eqref{e:estimate}
for a triple of smooth periodic functions. Indeed we remind the reader that, although $C^\infty$ is not dense in the strong topology of $C^\alpha$ (nor in that of $\Lip$), given a triple $(f, g, \varphi)\in C^\alpha\times W^{1,\infty} \times C^{1, \alpha}$ we can find a sequence $(f_k, g_k, \varphi_k)\in C^\infty \times C^\infty \times C^\infty$ such that:
\begin{itemize}
\item $\lim_k \|f_k -f\|_{C^0} = 0$ and $\|f_k\|_{C^\alpha}\leq \|f\|_{C^\alpha}$;
\item $g_k' \rightharpoonup^* g'$ in $L^\infty$ and $\|g_k\|_{C^\alpha} \leq \|g\|_{C^\alpha}$;
\item $\lim_k \|\varphi_k - \varphi\|_{C^0} = 0$ and $\|\varphi_k\|_{C^{1,\alpha}}\leq \|\varphi\|_{C^{1,\alpha}}$.
\end{itemize}
The conditions above are enough to infer
\[
\lim_{k\to \infty} \int f_k g_k' \varphi_k = \int f g' \varphi
\]
and thus it suffices to show that
\[
\left|\int f_k g_k' \varphi_k \right| \leq \|f_k\|_{C^\alpha}\|g_k\|_{C^\alpha} \|\varphi_k\|_{C^{1, \alpha}}\, .
\]
Fix therefore a triple $f,g, \varphi \in C^\infty (\mathbb S^1)$ and let 
\begin{align}
& f (x) = \sum_{k\in \mathbb Z} \hat{f}_k e^{ik\cdot x}\\
& g(x) = \sum_{k\in \mathbb Z} \hat{g}_k e^{ik\cdot x}\\
&\varphi (x) = \sum_{k\in \mathbb Z} \hat{\varphi}_k e^{ik\cdot x}\, 
\end{align}
be their Fourier expansions. 

We then know that the Fourier coefficients are necessarily real and that
\begin{equation}\label{e:Plancherel}
\int f g' \varphi = \sum_{(k, \ell) \in \mathbb Z^2} i (k-\ell) \hat{f}_\ell \hat{g}_{k-\ell} \hat{\varphi}_k\, .
\end{equation}
Recall next that, by Bernstein's inequality, $C^\alpha \subset H^\beta$ for every $\beta < \alpha$, thus 
\begin{align}
\sum_k (1+|k|^{2\beta})|\hat{f}_k|^2 & \leq C (\alpha, \beta) \|f\|_{C^\alpha}^2 \qquad \forall \beta < \alpha\\
\sum_k (1+|k|^{2\beta})|\hat{g}_k|^2 & \leq C (\alpha, \beta) \|g\|_{C^\alpha}^2\qquad \forall \beta < \alpha\, .
\end{align}
We finally need the simple estimate
\begin{equation}
|\hat{\varphi}_k|  \leq C \|\varphi\|_{C^{1,\alpha}} (1+ |k|)^{-1-\alpha}\label{e:decay}
\end{equation}
We are now ready to conclude and we start observing
\begin{align}
\left| \sum_\ell i (k-\ell) \hat{f}_\ell \hat{g}_{k-\ell}\right| & 
\leq |2k|^{1-\beta} \sum_{-k \leq \ell \leq k} |k-\ell|^\beta |\hat{g}_{k-\ell}||\hat{f}_\ell|
+ \sqrt{2} \sum_{\ell\leq -k, \ell \geq k} \sqrt{|k-\ell|} \sqrt{|\ell|} |\hat{g}_{k-\ell}| |\hat{f}_\ell|\nonumber\\
&\leq |2k|^{1-\beta} \left(\sum_j |j|^{2\beta} |\hat{g}_j|^2\right)^{\sfrac{1}{2}}
\left(\sum_j |\hat{f}_j|^2\right)^{\sfrac{1}{2}}\nonumber\\ 
&\qquad + \sqrt{2} \left(\sum_j |j| |\hat{g}_j|^2\right)^{\sfrac{1}{2}} \left(\sum_j |j| |\hat{f}_j|^2\right)^{\sfrac{1}{2}}\nonumber\\
&\leq C (1+ |k|)^{1-\beta} \|f\|_{C^\alpha} \|g\|_{C^{\alpha}}\, .\label{e:split_est}
\end{align}
Combining \eqref{e:Plancherel}, \eqref{e:decay} and \eqref{e:split_est} we then conclude
\begin{align}
\left|\int f g' \varphi\right| &\leq C \|f\|_{C^\alpha} \|g\|_{C^\alpha} \|\varphi\|_{C^{1,\alpha}} 
\sum_k (1+|k|)^{-\alpha-\beta} \leq C \|f\|_{C^\alpha} \|g\|_{C^\alpha} \|\varphi\|_{C^{1,\alpha}}\, ,
\end{align}
where we have used that, since we are free to choose any $\beta < \alpha$ and $\alpha > \frac{1}{2}$, we can impose $\alpha + \beta >1$, which ensures the convergence of the series $\sum_k (1+|k|)^{-\alpha-\beta}$. 
\end{proof}

\subsection{Connection}

Consider now a $C^2$ Riemannian manifold $(\Sigma, g)$ with $C^2$ boundary, a $C^2$ curve $\gamma: [a,b]\to \Sigma$ and a $C^1$ vector field along $\gamma$. In local coordinates we can write
\begin{align}
W (t) &= \sum_i W^i (t) \frac{\partial}{\partial x_i}\, , \\
\dot\gamma (t) &= \sum_i \dot\gamma^i (t) \frac{\partial}{\partial x_i}\, .
\end{align}
We then know that $\nabla_{\dot{\gamma}} W$ is given by the formula 
\begin{equation}\label{e:connection}
\frac{dW^i}{dt} \frac{\partial}{\partial x_i} +  \sum_{j, k} \Gamma^i_{jk} (\gamma)\, \dot\gamma^j\, W^k \frac{\partial}{\partial x_i}\, ,
\end{equation}
where the $C^1$ functions $\Gamma^i_{jk}$ are the Christoffel symbols of the metric $g$.

Let $u: \Sigma \to \mathbb R^m$ be a $C^{1,\alpha}$ isometric immersion. The vector field
$u_* W = \sum W^i \frac{\partial u}{\partial x_i}$ can thus be seen as a $C^\alpha$ map
$u_* W : [a,b] \to \mathbb R^m$. In particular, if $\alpha > \frac{1}{2}$ we can use Lemma \ref{l:product} to make sense
of the scalar product
\begin{equation}\label{e:product}
\left\langle \frac{d}{dt} u_* W,  \frac{\partial u}{\partial x_\ell}\right\rangle\, \, . 
\end{equation}
For smooth isometric immersions \eqref{e:product} and \eqref{e:connection} are then related by the identity
\begin{equation}\label{e:connection2}
\left\langle \frac{d}{dt} (u_* W (\gamma)),  \frac{\partial u}{\partial x_\ell} (\gamma)\right\rangle =
\sum_i \left(\frac{d}{dt} (W^i (\gamma)) + \sum_{j, k} \Gamma^i_{jk} (\gamma)\, \dot\gamma^j\, W^k (\gamma) \right) g_{i\ell} (\gamma)\, .
\end{equation}
The latter is just the classical relation between the Levi-Civita connection compatible with $g$ and the Levi-Civita connection
compatible with the standard Euclidean metric $e$ of the ambient Euclidean space. 
Lemma \ref{l:product} allows not only to make sense of the left hand side of the identity for $C^{1, \alpha}$ immersions when $\alpha > \frac{1}{2}$, but it also implies that, under the same regularity assumption, the identity \eqref{e:connection2} remains valid. 

\begin{proposition}\label{p:p=p}
Let $(\Sigma, g)$ be a $C^2$ Riemannian manifold with $C^2$ boundary, let $\gamma:[a,b]\to \Sigma$ be a $C^2$ curve, let $W$ be a 
$C^1$ vector field along $\gamma$ and let $u: \Sigma \to \mathbb R^m$ be an isometric immersion of class $C^{1,\alpha}$ for some $\alpha > \frac{1}{2}$. Then \eqref{e:connection2} holds. 
\end{proposition}

\begin{proof}[Proof of Theorem \ref{t:criticality}(a)]
The proposition implies part (a) of Theorem \ref{t:criticality} right away. Indeed, fix a point $p\in \partial \Sigma_a$ and choose local coordinates in a neighborhood $U$ of $p$ so that 
$X = \frac{\partial}{\partial x_2}$ on $U$ and $\frac{\partial}{\partial x_1}$ is tangent to $\Sigma_a$. Choose then $W$ tangent to $\Sigma_a$ and parametrize the curve $\gamma = \Sigma_a$ so that $\frac{d}{dt} v_* W = Z$. If we first use \eqref{e:connection2} for the standard embedding, we easily see that
\[
\sum_i \left(\frac{d}{dt} (W^i (\gamma)) + \sum_{j, k} \Gamma^i_{jk} (\gamma)\, \dot\gamma^j\, W^k\right) g_{i2} (\gamma) = a\, .
\]
If we then use it for $u=v$ we conclude
\[
\langle Y, Z \rangle = \left\langle \frac{\partial u}{\partial x_2} (\gamma), \frac{d}{dt} (u_* W (\gamma)) \right\rangle =
a\, .\qedhere
\]
\end{proof}

In order to prove the above proposition we recall the quadratic estimate in \cite[Proposition 1.6]{CDS10}: 
\begin{lemma}[Quadratic estimate]\label{l:CET}
Let $\Omega\subset \R^{n}$ be an open set, $v\in C^{1,\alpha}(\Omega, \R^{m})$ with $v^{\sharp}e\in C^{2}$ and $\varphi\in C^{\infty}(\R^{n})$ a standard symmetric convolution kernel. Then, for every compact set $K\subset \Omega$ 
 \[\| (v\ast \varphi_\epsilon)^* e -v^* e \|_{C^{1}(K)} = O(\epsilon^{2\alpha -1})\,.\]
\end{lemma}

\begin{proof}[Proof of Proposition \ref{p:p=p}] First observe that without loss of generality we can assume that $W$ is defined on the whole manifold. Secondly, observe that it suffices to prove the identity
for curves $\gamma$ which lie in the interior. Consider indeed a $C^2$ curve $\gamma$ which touches the boundary of the manifold and approximate it in $C^2$ with a sequence of
curves $\gamma_j$ which are contained in the interior. Then the maps $W (\gamma_j)$ converge in $C^1$ to $W (\gamma)$. As such, the maps
$u_* W(\gamma_j)$ are uniformly bounded in $C^\alpha$ and converge in $C^{\bar\alpha}$ to $u_* W (\gamma)$ for every $\bar\alpha < \alpha$. Since we can choose $\bar\alpha > \frac{1}{2}$, Lemma \ref{l:product} implies that the distributions
\[
\left\langle \frac{d}{dt} (u_* W (\gamma_j)), \frac{\partial u}{\partial x_\ell} (\gamma_j)\right\rangle
\]
converge to the distribution
\begin{equation}\label{e:still_prod}
\left\langle \frac{d}{dt} (u_* W (\gamma)), \frac{\partial u}{\partial x_\ell} (\gamma)\right\rangle\, .
\end{equation}
Moreover, obviously
\[
\frac{d}{dt} (W^i  (\gamma_j)) + \sum_{k,\ell} \Gamma^i_{k\ell} (\gamma_j) \dot{\gamma}_j^k W^\ell (\gamma_j)
\]
converge uniformly to
\begin{equation}\label{e:still_Chris}
\frac{d}{dt} (W^i (\gamma)) + \sum_{k,\ell} \Gamma^i_{k\ell} (\gamma) \dot{\gamma}^k W^\ell (\gamma)
\end{equation}
Fix now a curve $\gamma$ in the interior and a coordinate patch $U$ compactly contained in another coordinate patch $V$, both not intersecting the boundary of the manifold.
We can smooth $u$ by convolution with a standard kernel by $u\ast \varphi_\varepsilon$. For $\varepsilon$ small enough the convolution is well defined on the coordinate patch $U$.
Clearly the maps $(u\ast \varphi_\varepsilon)_* W$ and $(u\ast \varphi_\varepsilon)_* \frac{\partial}{\partial x_i}$ are uniformly bounded in $C^\alpha$ and converge, as $\varepsilon \downarrow 0$, to
$u_* W$ and $u_* \frac{\partial}{\partial x_i}$ in $C^\beta$ for every $\beta <\alpha$. Choosing a $\beta > \frac{1}{2}$ we apply Lemma \ref{l:product} to conclude that the distributions
\begin{equation}\label{e:approx}
\left\langle \frac{d}{dt} (((u\ast \varphi_\varepsilon)_* W) (\gamma)), \frac{\partial  (u\ast \varphi_\varepsilon)}{\partial x_i} (\gamma)\right\rangle
\end{equation}
converge (weakly in the sense of distributions) to \eqref{e:still_prod}. On the other hand, from Lemma \ref{l:CET}, if $\Gamma_{\varepsilon, k, \ell}^i$ denote the Christoffel symbols of the metric $(u\ast \varphi_\varepsilon)^* e$, then we conclude that they converge uniformly to $\Gamma^i_{k, \ell}$. Thus 
\begin{equation}\label{e:approx2}
\frac{d}{dt} (W^i (\gamma)) + \sum_{k, \ell} \Gamma_{\varepsilon, k, \ell}^i (\gamma) \dot\gamma^k W^\ell (\gamma)
\end{equation}
converge uniformly to \eqref{e:still_prod} and $[(u\ast \varphi_\varepsilon)^* e]_{ij}$ converges uniformly to $g_{ij}$. In particular, 
\begin{equation}\label{e:approx3}
\sum_i \left(\frac{d}{dt} (W^i (\gamma)) + \sum_{k, \ell} \Gamma_{\varepsilon, k, \ell}^i (\gamma) \dot\gamma^k W^\ell (\gamma)\right) [(u\ast \varphi_\varepsilon)^* e]_{i\ell} (\gamma)
\end{equation}
converge uniformly to the right hand side of \eqref{e:connection2}. 
However, since $u_\varepsilon$ is smooth, \eqref{e:approx} and \eqref{e:approx3} are equal by classical differential geometry. Letting $\varepsilon\downarrow 0$ we then conclude 
\eqref{e:connection2}. 
\end{proof}
\section{Flexibility: Proof of Theorem \ref{t:criticality} (B)} 
The maps $v$ violating the rigidity are produced by convex integration. Their construction relies on the following more general theorem, the proof of which is the content of most of the remaining sections.  

\begin{theorem} \label{t:main} Fix two integers $n\geq 2$, $m\geq n(n+2)$ and a metric $g\in C^{2}$ on $\bar B_1\subset \R^{n}$. There exists $\bar \sigma_0 >0$ such that if $u\in C^{\infty}(\bar B_1,\R^{m})$ and $h\in C^{\infty}(\bar B_1)$ are such that  
\begin{align} 
 & h\equiv h(|x|) >0 \text{ on } \mathring{B_1},\, h(1) = 0\, \text{ and } h'(1) \neq 0 \label{tmain:ass1}\\
 &\text{u is strictly short in $\mathring{B_1}$ and }\\
 & (1-\bar \sigma_0) he\leq g-u^{\sharp}e\leq (1+\bar \sigma_0)he\text{ in a neighborhood of $\partial B_1$}\label{tmain:ass2}\,,
\end{align}
then for every $\alpha <\frac{1}{2}$, every constant $x_0\in \R^{n(n+1)}$ and every $\varepsilon>0 $ there exists $v\in C^{1,\alpha}(\bar B_1, \R^{m+n(n+1)})$ such that  
\begin{align*}
&  \|v- (u,x_0)\|_{C^{0}(\bar B_1, \R^{m+n(n+1)})} < \varepsilon\,, \\
&  v=(u,x_0) \text{ and } \nabla v = (\nabla u\, 0) ^{\intercal} \text{ on }\partial B_1\\
& g= v^{\sharp}e\,.
\end{align*}
In addition, if $u$ is injective then $v$ can be chosen to be injective as well. 
\end{theorem}
If we manage to construct $h$ and $u$ satisfying \eqref{tmain:ass1}-\eqref{tmain:ass2} and, in addition, violating the rigidity at the boundary then we are done since the derivatives of $v$ and $u$ agree at the boundary.\\
Fix $R>1$ and consider the scaled spherical cap $\bar \Sigma_R\subset \R^{3}$ given as the image of $\Phi: \bar B_1\to \R^{3}$, where $\Phi(x_1,x_2) = (x_1,x_2,\sqrt{R^{2}-x_1^{2}-x_2^{2}}-\sqrt{R^{2}-1})$. 
We use polar coordinates to define the map $u: \bar B_1\to \R^{8}$ by 
\begin{equation}\label{d:shortembedding} u(r,\theta) = (\varphi(r)\cos\theta,\varphi(r)\sin\theta, 0,\ldots,0)\,,\end{equation}
where $\varphi \in C^{\infty}([0,1])$ is a suitable reparametrization such that $\varphi(0)=0,\, \varphi(1)=1,\, \varphi'(1)=\frac{R}{\sqrt{R^{2}-1}}$, and, for every $r\in ]0,1[$,
\begin{align}
 &\frac{R^{2}}{R^{2}-r^{2}}-\varphi'(r)^{2}>0 \,, \label{e:rshortness} \\ 
 &  r^{2}-\varphi(r)^{2}>0 \label{e:thetashortness}\,.
\end{align}
Observe that, once we produce such a $\varphi$, the map $u$ is strictly short in $\mathring B_1$ (except maybe in the origin, where the polar coordinates are not suited to the problem) and isometric on the boundary. Indeed, the  metric induced by $u$ is given in polar coordinates by 
\[ u^{\sharp}e = \varphi'^{2} dr^{2} +\varphi^{2}d\theta^{2}\,,\]
whereas the metric on $\Sigma_R$ which is induced by the inclusion into $\R^{3}$ reads 
\[ g= \frac{R^{2}}{R^{2}-r^{2}}dr^{2}+r^{2}d\theta^{2}\,.\]
Hence, the shortness away from the origin is given by \eqref{e:rshortness} and \eqref{e:thetashortness} whereas the isometry on the boundary is apparent from the values $\varphi(1)$ and $\varphi'(1)$. In the following, we construct a piecewise smooth function $\tilde \varphi$ satisfying the above assumptions; smoothing out the corners will then provide $\varphi$.
 We abbreviate $\gamma:= \frac{R}{\sqrt{R^{2}-1}}$. Because $R>1$ we can fix a positive $\eta \in ]2-\gamma, 1[$.  Since $\eta + \gamma >2$ we can then find $\varepsilon>0$ small enough such that 
\begin{equation}\label{beta} 
0<1-\varepsilon(\eta +\gamma)+\frac{\varepsilon^{2}}{2}(1-\frac{1}{\gamma}+\gamma^{3}R^{-2}) \leq (1-2\varepsilon)\left (1-\left (\varepsilon R^{-1}\right )^{2}\right )^{-\sfrac{1}{2}}\,,\end{equation}
as one can see by expanding $(1+x^{2})^{-\sfrac{1}{2}}$ around $x=0$. Set 
\[ \beta:= \frac{1-\varepsilon(\eta +\gamma)+\frac{\varepsilon^{2}}{2}(1 -\gamma^{-1}+\gamma^{3}R^{-2})}{1-2\varepsilon}\,,\] 
and define the piecewise continous 
\[ \phi(r) = \begin{cases} 
 \eta \,, & \text{ for } r\in [0,\varepsilon[ \\ 
 \beta \,, & \text{ for } r\in [\varepsilon,1-\varepsilon[\\
 \gamma - (1-\gamma^{-1}+\gamma^{3}R^{-2})(1-r) \,, &\text{ for } r\in [1-\varepsilon,1]\,.
\end{cases}\]
The definition of $\beta$ ensures that 
\begin{align*}\int_0^{1}\phi(r)dr &= \eta \varepsilon+\beta(1-2\varepsilon)+\varepsilon(\gamma-(1-\gamma^{-1}+\gamma^{3}R^{-2}))+\frac{1}{2}(1-\gamma^{-1}+\gamma^{3}R^{-2})\varepsilon(2-\varepsilon)\\ 
& = \varepsilon(\eta+\gamma) + (1-2\varepsilon)\beta -\frac{1}{2}(1-\gamma^{-1}+\gamma^{3}R^{-2})\varepsilon^{2}  = 1\,.
\end{align*}
Consequently, setting $\tilde \varphi(r)  =\int_0^{r}\phi(s)ds$ yields a continuous, piecewise smooth function with $\tilde \varphi(1)=1$ and $\tilde \varphi'(1)= \gamma = \frac{R}{\sqrt{R^{2}-1}}$. We claim that $\tilde \varphi$ satisfies \eqref{e:rshortness} and \eqref{e:thetashortness}. Indeed, on $]0,\varepsilon[$ this is provided by the fact that $\eta <1$. Moreover, if $\varepsilon$ is small enough then $\beta< 1$ which, together with \eqref{beta}, shows the inequalites on $[\varepsilon,1-\varepsilon[$. If $\varepsilon$ is small enough, \eqref{e:rshortness} holds on $]1-\varepsilon,1]$  since
\[ 
\left.\frac{d}{dr}\right|_{r=1}\left (\frac{R^{2}}{R^{2}-r^{2}} -\tilde \varphi '(r)^{2}\right )=\frac{2R^{2}}{(R^{2}-1)^{2}}-2\phi(1)\phi'(1) = 2(\gamma^{4}R^{-2}-\gamma(1-\gamma^{-1}+\gamma^{3}R^{-2})) <0\,,
\] 
and 
\[ \frac{R^{2}}{R^{2}-1}-\tilde \varphi'(1)^{2} =0\,.\]
Finally, on $[1-\varepsilon, 1]$ we have 
\[
\tilde{\varphi}' \geq \gamma - \varepsilon (1-\gamma^{-1} + \gamma^3 R^{-2})\, .
\]
In particular, for $\varepsilon$ small enough we have $\tilde{\varphi}' > 1$ on $[1-\varepsilon, 1]$. Since $\tilde{\varphi} (1)=1$, the latter implies that $\tilde{\varphi} (r) < r$ on $[1-\varepsilon, 1[$, thus concluding the proof of 
\eqref{e:thetashortness}. 

Consequently, if $u$ is defined by \eqref{d:shortembedding} then it is isometric on $\partial B_1$ and strictly short in $\mathring{B_1}\setminus\{0\}$. To show that it is also strictly short in the origin we switch to euclidean coordinates and observe that $u(x_1,x_2) = (\eta x_1,\eta x_2,0)$ if $|x| < \varepsilon$. Hence 
\[ g- u^{\sharp } e = \left (1-\eta^{2}+\frac{x_1^{2}}{R^{2}-|x|^{2}}\right )dx_1^{2}+\left (1-\eta^{2}+\frac{x_2^{2}}{R^{2}-|x|^{2}}\right )dx_2^{2}+2\frac{x_1x_2}{R^{2}-|x|^{2}}dx_1dx_2\,.\]
The shortness around the origin then again follows from $\eta<1$. Lastly, we define 
\[ h(r) = 2(\gamma-1)(1-r) \,.\]
Obviously, \eqref{tmain:ass1} is satisfied and we claim that, sufficiently close to $\partial B_1$, also \eqref{tmain:ass2} holds. For this we again consider the terms in polar coordinates. Expanding around $r=1$ gives 
\[ 1- \left ( \frac{\varphi}{r}\right )^{2} = 2(\gamma-1)(1-r)+o(|1-r|)\,,\] 
and 
\begin{align*} \frac{R^{2}}{R^{2}-r^{2}} -\varphi'^{2} &= \gamma^{2}+2\gamma^{4}R^{-2}(r-1) -\gamma^{2} -2\gamma(1-\gamma^{-1}+\gamma^{3}R^{-2})(r-1) +o(|r-1|) \\
&=2\gamma(r-1)(\gamma^{3}R^{-2}-(1-\gamma^{-1}+\gamma^{3}R^{-2}))+o(|r-1|)\\&= 2(\gamma-1)(1-r)  +o(|r-1|)\,.\end{align*}
This shows that 
\[ g-u^{\sharp}e -he = \left (\frac{R^{2}}{R^{2}-r^{2}}- \varphi'^{2}-h\right )dr^{2}+ r^{2}\left (1- \left( \frac{\varphi}{r}\right )^{2}-h\right )d\theta^{2} = o(|r-1|) e\,,\]
hence \eqref{tmain:ass2} is satisfied. Now fix $\alpha < \frac{1}{2}$. Then Theorem \ref{t:main} can be applied to find and isometric immersion  $v=(\underline v,w)\in C^{1,\alpha}\left (\bar B_1,\R^{8+6}\right )$ such that on $\partial B_1$ $\nabla \underline v = \nabla u$, $w= 0$ and $\nabla w = 0$.  

We now consider the appropriate rescaling of the map $v$ by $R$, namely $\frac{v}{R}$, which induces an isometric embedding of $\Sigma_a$ for $a= \sqrt{1- R^{-2}}$. Since the map is an isometry, the vector $Y = v_* X$ has the same length as the vector $X$, namely $|X|=1$. Observe, moreover, that by construction such vector field is in fact parallel to the vector field $Z$ and it has positive scalar product with it. In particular we conclude that $\langle Y, Z\rangle =1$. 

\section{Towards a Proof of Theorem \ref{t:main}: Main Iteration} 
The proof of Theorem \ref{t:main} is based on an iteration scheme developed by J. Nash in \cite{Nash54} to prove his counterintuitive result about the existence of $C^1$ isometric embeddings of $n$ dimensional manifolds into Euclidean space with suprisingly low codimension $n+1$. We need to adapt the scheme in two ways. First of all, in its original state it only produces maps which are $C^1$. Later renditions are able to get to $C^{1,\sfrac{1}{5}}$ in the case of two dimensional disks (see \cite{DIS15} and \cite{CDS10} for more general results). However, as realised in \cite{Kallen}, more regular isometric embeddings can be produced at the expense of increasing the codimension. Secondly, the iteration process needs to keep the boundary values fixed. This can be achieved, as done in \cite{Wasem}, by multiplying the perturbations by cutoff functions which are suited to the iteration scheme (see Lemma \ref{l:cutoff}). The following proposition is the main building block of the iteration.
\begin{proposition}\label{p:stage}
Let $n\geq2$, $m\geq n(n+2)$, $\lambda>0$ and fix an embedding $\tilde u\in C^{\infty}(\bar B_1, \R^{m})$. There exist constants $\sigma_0\in\, ]0,\frac{1}{2}[$, $R(\lambda) \geq 1$, $\Lambda(R)\geq 1$ and $C_0(\tilde u, \Lambda)\geq 1$ such that the following holds. Fix $c >b>1$ and 
\[ a > a_0(b,c,\sigma_0,\tilde u,  \lambda, R,\Lambda, C_0) \,,\]
and define 
\[ \delta_q = a^{-b^{q}},\quad \lambda_q = a^{cb^{q+1}}\,.\] 
Assume $\tilde g\in C^{2}$ is a metric on $\bar B_1$ with 
\begin{equation}\label{pstage:w_q}
[\tilde g]_k \leq C_0(1+\delta_{1}^{1-k})\, \quad\quad \text{for } k=0,1,2\,,
\end{equation}
and suppose $v_q\in C^{\infty}(\bar B_1, \R^{m})$ and $h_q \in C^{\infty}(\bar B_1)$ are such that 
\begin{align} 
 &v_q = \tilde u \text{ on } \bar B_1 \setminus B_{1-R\delta_{q+1}}\,, \,\, \| v_q - \tilde u \|_{1} < C_0\sum_{k=1}^{q}\delta_{k}^{\sfrac{1}{2}}\,, \, [v_q]_2 \leq C_0\delta_q^{\sfrac{1}{2}}\lambda_q\,, \label{pstage:v_q}\\
 &h_q \text{ is linear on } \bar B_1 \setminus B_{1-R\delta_{q+1}} \text{ with } h_q(1)=0,\, h_q'(1)=-\lambda\nonumber\\
&\qquad\qquad\qquad\qquad\text{ and } \,\Lambda^{-1}\delta_{q+1} \leq h_q \leq \Lambda \delta_{q+1} \,\text{ on } \bar B_{1-R\delta_{q+1}}\,, \label{pstage:h_q1} \\ 
 & [h_q]_k \leq C_0 \delta_{q+1}^{1-k} \text{ for } k=0,1,2,3 \,, \text{ and }\label{pstage:h_q2}\\
 & (1-\sigma_0(1+\eta_q))h_q e \leq \tilde g- v_q^{\sharp}e \leq (1+\sigma_0(1+\eta_q))h_q e \, \text{ on } \bar B_1 \,,\label{pstage:shortness}  
\end{align}
where $\eta_q \in C^{\infty}_c(\bar B_1)$ is a radially symmetric cutoff function with  $\eta_q \equiv 0 $ on $\bar B_1\setminus B_{1-R\delta_{q+1}}$, $\eta_q \equiv 1$ on $\bar B_{1-(R+1)\delta_{q+1}}$ and taking values between $0$ and $1$ (cf. Lemma \ref{l:cutoff} for the definition of the cutoffs). We can then find $v_{q+1},h_{q+1},\eta_{q+1}$ satisfying \eqref{pstage:v_q}--\eqref{pstage:shortness} with $q$ replaced by $q+1$ and, in addition, the following estimates hold:
\begin{align} 
 &\|v_{q+1}-v_q \|_0 \leq C_0\frac{\delta_{q+1}^{\sfrac{1}{2}}}{\lambda_{q+1}}\,,\label{pstage:c0}\\
 &[v_{q+1}-v_q]_1 \leq C_0 \delta_{q+1}^{\sfrac{1}{2}} \,.
 \label{pstage:c1}
\end{align}
\end{proposition}

\section{Proof of Proposition \ref{p:stage}: Preliminaries}
\subsection{H\"older spaces} In the following $m\in \N$, $\alpha\in ]0,1[$. The maps $f$ can be real-valued, vector-valued, matrix-valued or generally tensor-valued. In all these cases we endow the targets with the standard Euclidean norms, for which we will use the notation $|f (x)|$. 
We introduce the usual H\"older norms as follows.
First of all, the supremum norm is denoted by $\|f\|_0:=\sup |f|$.  We define the H\"older seminorms 
as
\begin{equation*}
\begin{split}
[f]_{k}&=\max_{|\beta|=m}\|D^{\beta}f\|_0\, ,\\
[f]_{k+\alpha} &= \max_{|\beta|=m}\sup_{x\neq y}\frac{|D^{\beta}f(x)-D^{\beta}f(y)|}{|x-y|^{\alpha}}\, .
\end{split}
\end{equation*}
The H\"older norms are then given by
\begin{eqnarray*}
\|f\|_{k}&=&\sum_{j=0}^{k}[f]_j\, ,\\
\|f\|_{k+\alpha}&=&\|f\|_k+[f]_{k+\alpha}.
\end{eqnarray*}
We then recall the standard ``Leibniz rule'' to estimate norms of products
\begin{equation}\label{e:Holderproduct}
[fg]_{r}\leq C\bigl([f]_r\|g\|_0+\|f\|_0[g]_r\bigr) \qquad \mbox{for any $1\geq r\geq 0$}
\end{equation}
and the usual interpolation inequalities
\begin{equation}\label{e:Holderinterpolation2}
[f]_{s}\leq C\|f\|_0^{1-\frac{s}{r}}[f]_{r}^{\frac{s}{r}} \qquad \mbox{for all $r\geq s\geq 0$}.
\end{equation}

We also collect two classical estimates on the H\"older norms of compositions. These are also standard, for instance
in applications of the Nash-Moser iteration technique. A proof can be found in \cite{DIS15}.

\begin{proposition}\label{p:chain}
Let  $\Psi: \Omega \to \mathbb R$ and $u: \R^n\supset U \to \Omega$ be two $C^{k}$ functions, with $\Omega\subset \R^N$. 
Then there is a constant $C$ (depending only on $k$,
$\Omega$ and $U$) such that
\begin{align}
\left[\Psi\circ u\right]_{k}&\leq C[u]_{k}\left( [\Psi]_1+\|u\|_0^{k-1}[\Psi]_k\right)\label{e:chain0}\, ,\\
\left[\Psi\circ u\right]_{k} &\leq C\left([u]_{k}[\Psi]_1+[u]_1^{k}[\Psi]_{k}\right)\label{e:chain1}\, .
\end{align} 
Let $f, g: \R^n\supset U \to \R$ two $C^{k}$ functions. Then there is a constant $C$ (depending only on $\alpha$, $k$, $n$ and $U$) such that
\begin{equation}\label{e:product1}
[fg]_{k} \leq C(\|f\|_0 [g]_{k} + \|g\|_0 [f]_{k})\, .
\end{equation}
\end{proposition}

\subsection{Quadratic mollification estimate} 
We will often use regularizations of maps $f$ by convolution with a standard mollifier $\varphi_\ell (y) := \ell^{-n} \varphi (\frac{y}{\ell})$, where $\varphi\in C^\infty_c (B_1)$ is assumed to have integral $1$ and to be non negative and rotationally symmetric. We will need the following estimates. For a proof see \cite{CDS10}.

\begin{lemma}\label{l:mollify}
For any $r,s\geq 0$ and $0<\alpha\leq 1$ we have
\begin{align}
&[f*\varphi_\ell]_{r+s}\leq C\ell^{-s}[f]_r,\label{e:mollify1}\\
&[f-f*\varphi_\ell]_r \leq C\ell^{2}[f]_{2+r},\label{e:mollify2}\\
&\|f-f*\varphi_\ell\|_r \leq C\ell^{2-r}[f]_{2}, \qquad \mbox{if $0\leq r\leq2$} \label{e:mollify4}\\
&\|(fg)*\varphi_\ell-(f*\varphi_\ell)(g*\varphi_\ell)\|_r\leq C\ell^{2\alpha -r}\|f\|_\alpha\|g\|_\alpha\, ,\label{e:mollify3}
\end{align}
where the constants $C$ depend only upon $s$, $r$, $\alpha$ and $\varphi$. 
\end{lemma}

\subsection{Existence of normals}
The following proposition claims the existence of an orthonormal family of normal vectorfields to the embedded surface together with the appropriate estimates \eqref{pnormals3}. It is already contained in \cite{Kallen}, but our condition on the co-dimension is less restrictive ($d\geq 1$ as opposed to $d\geq n+1$). The reason for this is that in the proof we use Lemma \ref{l: normals} below instead of Lemma 2.5 of \cite{Kallen}. The rest of the proof is essentially unchanged. For the readers convenience we provide the details in the appendix. 

\begin{proposition}\label{p:normals} Let $n\geq 2$, $d\geq 1$, $B$ a set diffeomorphic to the closed unit ball of $\R^{n}$ and $u\in C^{\infty}\left (B,\R^{n+d}\right )$ an immersion. There exists $\rho_0 \equiv \rho_0(d,n,u) >0$ and constants $C_k$ depending only on $u$ such that the following holds. If $v\in C^{\infty}\left (B,\R^{n+d}\right )$ is such that 
\[ \| v-u \|_{C^{1}}  < \rho_0\,,\]
then there exist $\zeta_1(v),\ldots, \zeta_d(v) \in C^{\infty}
\left (B,\R^{n+d}\right )$ such that for all $1\leq i,j\leq d$ we have
\begin{align}
\langle \zeta_i(v),\zeta_j(v)\rangle& = \delta_{ij} \quad \text{ on } B\label{pnormals1}\\
\nabla v \cdot \zeta_i(v) &= 0 \quad \quad \text{ on } B\label{pnormals2}
\end{align}
and 
\begin{equation} \label{pnormals3}
[\zeta_i(v)]_{k} \leq C_k(1+\|v\|_{k+1})\,.
\end{equation}
\end{proposition}

\subsection{Decomposition of the metric error}
We use the following decomposition of the metric error, in the spirit of Lemma 2.3 in \cite{Kallen}. The proof is a simple application of the implicit function theorem and is provided in the appendix. 
\begin{proposition} \label{p:decomp}
There exists $r_0>0$ and $\nu_1,\ldots,\nu_{n_*}\in\mathbb{S}^{n-1}$ with the following property. If $\tau :\bar B_1\to\text{Sym}_n^{+}$ and $\displaystyle \{M_i\}_{i=1,\ldots,n_*}$, $\displaystyle \{\Lambda_{ij}\}_{i,j=1,\ldots,n_*}\subset C^{\infty}(\bar B_1, \text{Sym}_n)$ are such that 
\[ \| \tau -\text{Id}\|_0 + \sum_{i=1}^{n_*} \|M_i\|_0+ \sum_{i,j=1}^{n_*}\|\Lambda_{ij}\|_0 < r_0 \,,\]
then there exist smooth functions $c_1,\ldots,c_{n_*}: \bar B_1\to \R$ with 
\begin{equation}\label{e:pdecomp1}
 \forall x\in \bar B_1 : \quad \tau(x) = \sum_{i=1}^{n_*}c_i^{2}(x) \nu_i\otimes \nu_i + \sum_{i=1}^{n_*}c_i(x) M_i(x) + \sum_{i,j=1}^{n_*} c_i(x)c_j(x) \Lambda_{ij}(x)\,,
\end{equation}
$c_i(x) > r_0$ on $\bar B_1$, and for any $\Omega\subset \bar B_1$
\begin{equation}\label{e:pdecomp2} 
\| c_i \|_{k,\Omega} \leq C_{k}\left (1 + \|\tau\|_{k,\Omega} + \sum_{i=1}^{n_*}\|M_i\|_{k,\Omega}  + \sum_{i,j =1}^{n_*} \|\Lambda_{ij}\|_{k,\Omega}\right ) \,.
\end{equation}
\end{proposition}

\subsection{Cutoff functions} In order to keep the boundary values the same along the iteration we will multiply the perturbations with a suitable cutoff function. The following lemma clarifies the type of cutoff we will use and its most important properties. 
\begin{lemma}\label{l:cutoff} There exist universal constants $\varepsilon>0, C\geq 1$ and a sequence of radially symmetric cutoff functions $\left (\eta_q\right)_{q\in\N} \subset  C^{\infty}_c\left (\bar B_1\right )$ such that for any $q\in\N$ we have 
\begin{align}
 &\eta_q \equiv 1 \text{ on } \bar B_{1-(R+1)\delta_{q+1}} \,\text{ and }\, \eta_q \equiv 0 \text{ on } \bar B_1\setminus B_{1-R\delta_{q+1}}\,,\\ 
 & [\eta_q]_k \leq C\delta_{q+1}^{-k}\, \text{ for } k\geq 0\,,\label{e:etaderivatives}\\ 
 & \eta_q\leq \varepsilon \Rightarrow |\nabla \eta_q ^{\intercal}\nabla \eta_q| \leq C\delta_{q+1}^{-2} \eta_q\,. \label{e:eta'smallness}
\end{align}
\end{lemma}

\begin{proof} 
Define $f\in C^{0}(\R)$ by $f\equiv 0$ on $]-\infty,\frac{1}{4}]$, $f\equiv 1$ on $[\frac{3}{4},+\infty[$ and linear in between. Smoothing out the corners by mollifying $f$ with a standard mollifying kernel $\varphi_\ell$ with parameter $\ell <\frac{1}{4}$ we find a function $h= f\ast \varphi_\ell\in C^{\infty}(\R)$ satisfying $h\equiv 0$ on $]-\infty,0]$ and $h\equiv 1$ on $[1,+\infty[$. Also, since $h''(r) \to 0$ as $r\to 0$, we can find $\varepsilon>0$ such that 
\[ h\leq\varepsilon \Rightarrow (h')^{2} \leq h\,.\]
The sequence $\eta_q$ is then easily constructed by setting, for $x\in \bar B_1$, 
\[ \eta_q(x) := h\left (\delta_{q+1}^{-1}\left (1-R\delta_{q+1}-|x|\right )\right )\,.\qedhere\]
\end{proof}

\subsection{Parameters} To counteract the loss of derivatives appearing along the iteration we mollify the map by convolution with a standard kernel so that we can control higher derivatives with the mollification parameter $\ell$. However, we have to make sure that this parameter is chosen small enough to keep the metric error \eqref{pstage:shortness} of the same size. It turns out that the right choice is 
\begin{equation}\label{d:l}
\ell := \frac{1}{\tilde C} \frac{\delta_{q+1}^{\sfrac{1}{2}}}{\delta_q^{\sfrac{1}{2}}\lambda_q}\,,
\end{equation}
where $\tilde C\geq 1$ is a universal constant, depending additionally on $\tilde u$, $g$, $R$, $\Lambda$ and $C_0$, which will be chosen in Lemma \ref{l:decompose}. In the course of the proof we will need the following hierarchy of the parameters 
\begin{equation}\label{e:parameters1}
\delta_{q+1}^{-1}\leq \delta_{q+2}^{-1}\leq \ell^{-1}\leq \lambda_{q+1}\,.
\end{equation}
The first inequality is true by definition, while the second follows from 
\begin{align*}
\log_a(\delta_{q+2}\ell^{-1}) & =\log_a\left (\tilde C\delta_q\delta_{q+1}^{-\sfrac{1}{2}}\delta_{q+2}\lambda_q\right ) >-\frac{1}{2}b^{q}+\left (c+\frac{1}{2}\right )b^{q+1}-b^{q+2}\\
&= b^{q}\left (\frac{1}{2}(b-1)+b(c-b)\right )>0\,.
\end{align*} 
In particular, we also have 
\begin{equation}\label{e:parameters2}
\delta_{q+1}^{-\sfrac{1}{2}}\leq \delta_q^{\sfrac{1}{2}}\lambda_q\,.
\end{equation}
The last inequality in \eqref{e:parameters1} is a consequence of the following stronger estimate, which will be needed in Section \ref{s:conclusion}. Fix any constant $\hat C(b,c,\sigma_0,\tilde u, g,\lambda,R,\Lambda,C_0)$. Then, if $a\geq a_0(\hat C)$ is chosen large enough, we have 
\begin{equation}\label{e:errorsize}
\hat C \frac{\delta_{q+1}}{\ell^{2}\lambda_{q+1}^{2}}\leq \delta_{q+2}\,.
\end{equation}
Indeed, inserting the definition of $\ell$ we see that the inequality is satisfied if 
\[ \hat C^{-1}\tilde C^{-2} \delta_q^{-1}\lambda_q^{-2}\delta_{q+2}\lambda_{q+1}^{2}\geq 1\,.\]
Taking the logarithms gives 
\[ b^{q}\left (b^{2}(2c-1)-2bc+1\right )-\log_a\left (\hat C\tilde C^{2}\right )\geq 0\,.\]
Rewriting the first term, we find
\[ b^{q}(b-1)\left (b(2c-1)-1\right )-\log_a\left (\hat C\tilde C^{2}\right )\geq 0\,.\]
This inequality is satisfied if $a$ is chosen large enough, so that \eqref{e:errorsize} holds. 
\section{Proof of Proposition \ref{p:stage}: Setup}

\subsection{Mollification} Fix a standard, symmetric mollifier, i.e. a radially symmetric, nonnegative function $\varphi\in C^{\infty}_c(B_1)$ on $\R^{n}$ with unit integral and set $\varphi_\ell(x) = \ell^{-n}\varphi(x/\ell)$.
We define the mollification parameter  $\ell$ by \eqref{d:l} and set 
\begin{equation} \label{d:barv_q} 
\bar v_q := (v_q- \tilde u)\ast \varphi_\ell + \tilde u \,,
\end{equation} 
which mollifies the map $v_q$ while keeping the boundary value: since $\delta_q^{\sfrac{1}{2}}\lambda_q > \delta_{q+1}^{-\sfrac{1}{2}} $ we have
 $\ell < \frac{1}{2}R\delta_{q+1}$ if $\tilde C$ is chosen large enough, so that, thanks to \eqref{pstage:v_q}, it holds
\[ \bar v_q = \tilde u \text{ on } \bar B_1 \setminus B_{1-\frac{1}{2}R\delta_{q+1}}\,.\]
Lastly, we set
 \begin{equation}\label{d:tau}
 \tau := \frac{\tilde g- \bar v_q ^{\sharp}e}{h_q}-\frac{\delta_{q+2}}{h_q}e \,.
 \end{equation} 
Observe that $\tau$ is welldefined and smooth on every compactly contained $\Omega\subset B_1$. We gather a few important estimates on $\bar v_q$ and $\tau$ in the next
\begin{lemma}\label{l:decompose} If $\tilde C(\tilde u,\Lambda,C_0)$, $a_0(C_0,\Lambda)$ and $R(\lambda)$ are chosen large enough and if $\sigma_0>0$ is chosen small enough, then, for $k=0,1,2$, we have 
\begin{align} 
&[\bar v_q ]_{k+1} \leq C(1+\delta_{q+1}^{\sfrac{1}{2}}\ell^{-k})\,,\label{e:barv_qk}\\
&[\bar v_q^{\sharp}e - v_q^{\sharp}e\ast \varphi_\ell]_{k}\leq C\ell^{2-k}[v_q]_2^{2}\,,\\
& |\tau-e| \leq  \frac{r_0}{2}\, \text{ on } \bar B_{1-R\delta_{q+2}}\,,\label{e:tau0}\\ 
 &|D^{k}\tau| \leq C\ell^{-k} \, \text{ on } \bar B_{1-R\delta_{q+2}}\label{e:tauk} \,,
\end{align}
for some constant $C$ depending on $\tilde u$ and $\Lambda$. 
\end{lemma}

\begin{proof}
First observe that if $a_0(C_0)$ is large enough we get $\|v_q\|_1 \leq \|\tilde u\|_1 +1 \leq C(\tilde u)$.  Therefore, using again \eqref{pstage:v_q} and Lemma \ref{l:mollify},
\[ [\nabla \bar v_q]_k = [\nabla v_q\ast \varphi_\ell]_k +[\nabla(\tilde u-\tilde u\ast\varphi_\ell)]_k \leq C(\tilde u)(1+\ell^{1-k}[v_q]_2)+ C\ell^{1-k}[\tilde u]_2 \leq C(\tilde u)(1+\delta_{q+1}^{\sfrac{1}{2}}\ell^{-k})\,,\]
if $\tilde C(C_0)$ is large enough. For the second estimate we compute 
\[ \nabla \bar v_q ^{\intercal}\nabla \bar v_q = \nabla(v_q\ast\varphi_\ell)^{\intercal}\nabla(v_q\ast \varphi_\ell) + \nabla(\tilde u-\tilde u\ast\varphi_\ell)^{\intercal}\nabla(\tilde u-\tilde u\ast \varphi_\ell) +2 \text{sym}\left (\nabla(v_q\ast\varphi_\ell)^{\intercal}\nabla(\tilde u-\tilde u\ast\varphi_\ell)\right )\,,\]
where we denoted sym$(A) = \frac{1}{2}(A+A^{\intercal})$. This gives 
\begin{align*} [\bar v_q^{\sharp}e - v_q^{\sharp}e\ast \varphi_\ell]_{k}&\leq C(\tilde u)\Big( [(v_q\ast \varphi_\ell)^{\sharp}e- v_q^{\sharp}e\ast \varphi_\ell]_k +(1+[\tilde u-\tilde u\ast\varphi_\ell]_1)[\tilde u-\tilde u\ast\varphi_\ell]_{k+1}\\
&\qquad\qquad+[v_q\ast\varphi_\ell]_{k+1}[\tilde u-\tilde u\ast\varphi_\ell]_1\Big)\\ &\leq C(\tilde u)\left (\ell^{2-k}[v_q]_2^{2}+\ell^{2-k}[\tilde u]_3+(1+\ell^{1-k}[v_q]_2)\ell^{2}[\tilde u]_3\right )\leq C(\tilde u)\ell^{2-k}[v_q]_2^{2}\,. \end{align*}
We will prove the estimates \eqref{e:tau0} and \eqref{e:tauk} separately on $\bar B_{1-R\delta_{q+2}}\setminus B_{1-\frac{1}{2}R\delta_{q+1}}$ and on $\bar B_{1-\frac{1}{2}R\delta_{q+1}}$. Since on the former we have $\bar v_q = \tilde u= v_q$, and consequently
\[ \tau - e= \frac{\tilde g- v_q^{\sharp }e-  h_q e}{h_q} -\frac{\delta_{q+2}}{h_q}e\,, \] 
 it follows with \eqref{pstage:shortness} and $h_q \geq \lambda R\delta_{q+2}$ that 
 \[ |\tau -e| \leq C\sigma_0 +C\frac{1}{\lambda R} \leq \frac{r_0}{2}\,\]
 if $\sigma_0$ is small and $R (\lambda)$ large enough. By \eqref{pstage:shortness} we have the pointwise estimate $| \tilde g- v_q^{\sharp}e| \leq C|h_q|$, so that with the help of  \eqref{pstage:w_q} and \eqref{pstage:h_q2}
 \[ |\nabla \tau | \leq C\left (\frac{|\nabla(\tilde g -v_q^{\sharp}e)|}{h_q}+ \frac{|\nabla h_q|}{h_q}\right )\leq C(\tilde u)C_0\delta_{q+2}^{-1}\,, \]
 and similarly 
 \begin{align*} | D^{2}\tau|&\leq C\left ( \frac{|D^{2}h_q|}{h_q} + \frac{|\nabla h_q|\left (|\nabla h_q|+|\nabla(\tilde g-v_q^{\sharp}e)|\right ) }{h_q^{2}}+\frac{|D^{2}(\tilde g-v_q^{\sharp}e)|}{h_q}\right )\\
 & \leq C(\tilde u)C_0\left (\delta_{q+1}^{-1}\delta_{q+2}^{-1}+\delta_{q+2}^{-2}+\delta_{q+1}^{-1}\delta_{q+2}^{-1}\right )\leq C(\tilde u)C_0\delta_{q+2}^{-2}\,.\end{align*}
Observe that, if $\tilde C\geq C_0$ then $ C_0 \delta_{q+2}^{-k} \leq \ell^{-k}$ for $k=1,2$, thanks to \eqref{e:parameters1}. This shows \eqref{e:tauk} on $\bar B_{1-R\delta_{q+2}}\setminus B_{1-\frac{1}{2}R\delta_{q+1}}$. To show the estimates on $\bar B_{1-\frac{1}{2}R\delta_{q+1}}$ we write 
\begin{align*}
|\tau -e | &\leq C\frac{\delta_{q+2}}{\Lambda^{-1}\delta_{q+1}}+\frac{1}{h_q}\big| ( \tilde g-v_q^{\sharp}e-h_q e)\ast\varphi_\ell+(v_q^{\sharp}e\ast\varphi_\ell-\bar v_q^{\sharp}e) +(h_q\ast\varphi_\ell-h_q)e\\
&\qquad\qquad\qquad\qquad  +(\tilde g-\tilde g\ast\varphi_\ell)\big| \\ 
& \leq \frac{r_0}{8}+\frac{C}{h_q}\left (\sigma_0 |h_q\ast\varphi_\ell| + \ell^{2}([v_q]_2^{2}+[h_q]_2+[\tilde g]_2)\right )\\
&\leq \frac{r_0}{8}+ C\sigma_0 + \frac{C\ell^{2}}{h_q}\left( C_0^{2}\delta_q\lambda_q^{2}+C_0\delta_{q+1}^{-1} + C_0(2+\sigma_0)\delta_{q+1}^{-1}\right )\\ 
&\leq \frac{r_0}{4}+C\frac{C_0^{2}}{\tilde C^{2}\Lambda^{-1}} \leq \frac{r_0}{2}\,
\end{align*}
if $\sigma_0$ is chosen small and $\tilde C(\Lambda,C_0)$ as well as $a(\Lambda)$ large enough. This fixes the choice of $\tilde C$. For \eqref{e:tauk} we estimate 
\begin{align*} [\tilde g-\bar v_q^{\sharp}e]_k &\leq [(\tilde g-v_q^{\sharp}e)\ast\varphi_\ell]_k +[\tilde g-\tilde g\ast\varphi_\ell]_k +[v_q^{\sharp}e\ast\varphi_\ell-\bar v_q^{\sharp}e]_k  \\
& \leq C(\tilde u) \left (\ell^{-k}\|\tilde g-v_q^{\sharp}e\|_0 +\ell^{2-k}([\tilde g]_2+[v_q]_2^{2})\right )\leq C(\tilde u,\Lambda)\delta_{q+1}\ell^{-k}\,.
\end{align*}
Hence, with the help of \eqref{e:chain0} we get on  $\bar B_{1-\frac{1}{2}R\delta_{q+1}}$ 
\begin{align*} |D^{k}\tau| &\leq C\left (\Lambda\delta_{q+1}^{-1}[\tilde g-\bar v_q^{\sharp}e]_k+[h_q]_k(\Lambda^{2}\delta_{q+1}^{-2} + (\Lambda\delta_{q+1})^{k-1}(\Lambda^{-1}\delta_{q+1})^{-k-1})(\|\tilde g-\bar v_q^{\sharp}e\|_0+\delta_{q+2})\right )\\
 & \leq C(\tilde u, \Lambda)\left (\ell^{-k} +C_0\delta_{q+1}^{-k}\right )\leq C(\tilde u, \Lambda)\ell^{-k}\,.\qedhere\end{align*}
\end{proof} 

\subsection{Decomposition} 
Our goal in constructing $v_{q+1}$ is to add the (rescaled) metric error $\tau$ by an ansatz of the form 
\begin{equation}\label{d:vq+1} v_{q+1} = \bar v_q + \sum_{k=1}^{n_*} \frac{a_k}{\lambda_{q+1}}\left (\sin(\lambda_{q+1} \nu_k\cdot x)\zeta^{1}_k + \cos(\lambda_{q+1} \nu_k\cdot x)\zeta^{2}_k\right )\,,\end{equation}
where $\nu_k\in\S^{n-1}$, $a_k$ are smooth coefficients and where $\zeta^{1}_k,\zeta^{2}_k$ are smooth, mutually orthogonal unit vector fields which are normal to $\bar v_q$. 
We compute 
\begin{align} \nabla v_{q+1} = \nabla \bar v_q &+ \sum_{k=1}^{n_*}a_k\underbrace{\left (\cos(\lambda_{q+1} \nu_k\cdot x)\zeta^1_k\otimes \nu_k - \sin(\lambda_{q+1} \nu_k\cdot x)\zeta^2_k\otimes \nu_k\right )}_{=:A_k}\nonumber\\
&+ \sum_{k=1}^{n_*}\frac{a_k}{\lambda_{q+1}}\underbrace{\left (\sin(\lambda_{q+1} \nu_k\cdot x) \nabla\zeta^1_k + \cos(\lambda_{q+1} \nu_k\cdot x)\nabla\zeta^2_k\right )}_{=:B_k}\nonumber  \\ 
&+\sum_{k=1}^{n_*}\frac{1}{\lambda_{q+1}}\underbrace{\left (\sin(\lambda_{q+1} \nu_k\cdot x)\zeta^1_k + \cos(\lambda_{q+1} \nu_k\cdot x)\zeta^2_k\right )}_{=:C_k}\nabla a_k \label{e:dvq+1}\,, 
\end{align}
so that (in coordinates) the induced metric is 
\begin{align}
\nabla v_{q+1}^{\intercal}\nabla v_{q+1} = \nabla \bar v_q^{\intercal}&\nabla \bar v_q +\sum_{k=1}^{n_*}a_k^{2}\nu_k\otimes \nu_k +2\sum_{k=1}^{n_*}\frac{a_k}{\lambda_{q+1}}\text{sym}(\nabla \bar v_q ^{\intercal} B_k) +2\sum_{i,j=1}^{n_*}\frac{a_ia_j}{\lambda_{q+1}} \text{sym}(A_i^{\intercal}B_j) \nonumber \\ &+2\sum_{i,j=1}^{n_*}\frac{a_ia_j}{\lambda_{q+1}^{2}}\text{sym}\left (B_i^{\intercal}B_j\right )+ 2\sum_{i,j=1}^{n_*}\frac{a_i}{\lambda_{q+1}^{2}}\text{sym}(B_i^{\intercal}C_j\nabla a_j)+  \sum_{k=1}^{n_*}\frac{1}{\lambda_{q+1}^{2}}\nabla a_k^{\intercal }\nabla a_k\label{e:newinducedmetric}\,.
\end{align}
The usual practice is to decompose the metric error $\tilde g-\bar v_q^{\sharp}e$ into a sum of the form $\sum_{k=1}^{n_*}a_k^{2}\nu_k\otimes \nu_k$ and hence the ansatz \eqref{d:vq+1} allows the addition of the metric error upto errors which are (if $\lambda_{q+1}$ is chosen large) very small.  However, as realized in \cite{Kallen}, a better convergence rate is achieved if only the terms in the second line of \eqref{e:newinducedmetric} are treated as error terms. Consequently, one needs a slightly subtler decomposition, which is provided by Proposition \ref{p:decomp} once we know that the first error terms are small enough. This is the content of Lemma \ref{l:ugliestlemma} once we have found suitable normal vectors $\zeta_k^{1}, \zeta_k^{2}$. But this is an easy task thanks to Proposition \ref{p:normals}, once we require $a(\tilde u,C_0)$ to be so large that $C_0\sum_{k=1}^{q}\delta_q^{\sfrac{1}{2}}<\rho_0(\tilde u)$, where $\rho_0$ is given by Proposition \ref{p:normals}. Then, since 
\[ \|\bar v_q- \tilde u\|_1 = \|(v_q-\tilde u)\ast\varphi_\ell\|_1 \leq \|v_q-\tilde u\|_1< \rho_0(\tilde u)\,,\]
Proposition \ref{p:normals} provides an orthonormal family of vectorfields $\xi_1(\bar v_q),\ldots\xi_{m-n}(\bar v_q) \in C^{\infty}(\bar B_1,\R^{m})$ which are normal to $\bar v_q$ and enjoy the estimates 
\begin{align}  |D^{k}\xi_i|&\leq C(\tilde u) \, \text{ on } \bar B_{1-R\delta_{q+2}}\setminus B_{1-\frac{1}{2}R\delta_{q+1}} \label{DknormalsOmega} \\
|D^{k}\xi_i|& \leq C(\tilde u)(1+\delta_{q+1}^{\sfrac{1}{2}}\ell^{-k}) \, \text{ on } \bar B_{1-\frac{1}{2}R\delta_{q+1}}\,,\label{Dknormals}
\end{align}
for $k=0,1,2,$ thanks to \eqref{e:barv_qk}. We now define 
\begin{equation}\label{d:normals}
\zeta^{1}_i := \xi_i\, \,, \zeta^{2}_i := \xi_{n_*+i} \,,\quad \text{ for } i =1,\ldots,n_*\,, 
\end{equation}
which is possible in view of $m-n\geq n(n+2)-n =2n_*$. \\
Now let $\nu_1,\ldots, \nu_{n_*}$ be the vectors given by Proposition \ref{p:decomp},  define $A_k$, $B_k$ and $C_k$ as in \eqref{e:dvq+1}, let $\eta:=\eta_{q+1}$ be one of  the cutoff functions constructed in Lemma \ref{l:cutoff} and set 
\begin{align} 
M_i&:= \frac{2}{h_q^{\sfrac{1}{2}}\lambda_{q+1}}\text{sym}\left (\nabla \bar v_q^{\intercal}B_i\right )\label{d:Mi}\\
\Lambda_{ij}&:= \frac{2}{\lambda_{q+1}}\text{sym}\left (A_i^{\intercal}B_j\right )+ \frac{2}{\lambda_{q+1}^{2}}\text{sym}\left (B_i^{\intercal}(B_j+C_j\nabla\eta)\right )+\frac{2}{h_q^{\sfrac{1}{2}}\lambda_{q+1}^{2}}\text{sym}\left (B_i^{\intercal}C_j\nabla h_q^{\sfrac{1}{2}}\right ) \nonumber \\
&\,\,\,\,\,\,+ \frac{\delta_{ij}}{\lambda_{q+1}^{2}}\nabla \eta^{\intercal}\nabla \eta +\frac{2\delta_{ij}}{h_q^{\sfrac{1}{2}}\lambda_{q+1}^{2}}\text{sym}\left (\nabla \eta^{\intercal}\nabla h_q^{\sfrac{1}{2}}\right )+\frac{\delta_{ij}}{h_q\lambda_{q+1}^{2}}\nabla( h_q^{\sfrac{1}{2}})^{\intercal}\nabla h_q^{\sfrac{1}{2}}\,.\label{d:Lambdaij}
\end{align}

\begin{lemma} \label{l:ugliestlemma} For $a(b,c,\tilde u,\lambda,R,C_0)$ large enough there exists a constant $C>0$ (depending only on $\tilde u$ and $\Lambda$) such that for $k=0,1,2$ 
\begin{align}
&|D^{k}A_i| +| D^{k}C_i|\leq C\lambda_{q+1}^{k}\,\, \text{ on } \bar B_{1-R\delta_{q+2}}\,,\\ 
&|D^{k}B_i| \leq C\delta_{q+1}^{\sfrac{1}{2}}\ell^{-1}\lambda_{q+1}^{k}\,\, \text{ on } \bar B_{1-\frac{1}{2}R\delta_{q+1}} \,\, \text{ and } \,|D^{k}B_i| \leq C\lambda_{q+1}^{k} \,\,\text{ on }\, \bar B_{1-R\delta_{q+2}}\setminus B_{1-\frac{1}{2}R\delta_{q+1}}\,,\label{e:Bi}\\
&|D^{k}M_i| + |D^{k}\Lambda_{ij}|\leq C \ell^{-1}\lambda_{q+1}^{k-1} \,\, \text{ on } \bar B_{1-R\delta_{q+2}}\,.
\end{align}
\end{lemma}

\begin{proof}
Since the vectors $\nu_k$ are constant, the estimate for $A_i$ and $C_i$ is (up to a constant) the same: 
\[|D^{k}C_i| \leq C\left (\lambda_{q+1}^{k}+[\zeta_i^{j}]_k\right ) \leq C\left (\lambda_{q+1}^{k}+C(\tilde u)\delta_{q+1}^{\sfrac{1}{2}}\ell^{-k}\right ) \leq C\lambda_{q+1}^{k}\,,\]
if $a(\tilde u)$ is large enough, where we have used $\lambda_{q+1}\geq \ell^{-1}$. The estimate for $B_i$ follows from 
\[ |D^{k}B_i| \leq C(\lambda_{q+1}^{k}[\zeta_i^{j}]_1 +[\zeta_i^{j}]_{k+1}) \] 
using \eqref{DknormalsOmega} and \eqref{Dknormals} respectively. Since $h_q \geq R\lambda\delta_{q+2}\geq \delta_{q+2}$ on $\bar B_{1-R\delta_{q+2}}$ and $h_q \geq \Lambda^{-1}\delta_{q+1}$ on $\bar B_{1-\frac{1}{2}R\delta_{q+1}}$ we get, using \eqref{e:chain0}, 
\[ |D^{k+1}h_q^{\sfrac{1}{2}} | \leq C(\Lambda)C_0\delta_{q+1}^{-k}\left (\delta_{q+1}^{-\sfrac{1}{2}}+\delta_{q+1}^{k}\delta_{q+1}^{-\sfrac{1}{2}-k}\right ) \leq C(\Lambda)C_0\delta_{q+1}^{-\sfrac{1}{2}-k}\,\]
on $\bar B_{1-\frac{1}{2}R\delta_{q+1}}$, and 
\[ |D^{k+1}h_q^{\sfrac{1}{2}} | \leq C(\lambda,R)C_0\delta_{q+1}^{-k}\left (\delta_{q+2}^{-\sfrac{1}{2}}+\delta_{q+1}^{k}\delta_{q+2}^{-\sfrac{1}{2}-k}\right )   \leq C(\lambda,R)C_0\delta_{q+2}^{-\sfrac{1}{2}-k}\, \]
on $\bar B_{1-R\delta_{q+2}}\setminus B_{1-\frac{1}{2}R\delta_{q+1}}$. Now, combining \eqref{e:Bi} and the previous two estimates,

\begin{align*} |D^{k}M_i|& \leq \frac{C}{\lambda_{q+1}}\left ([h_q^{-\sfrac{1}{2}}]_k \|B_i\|_0+|h_q^{-\sfrac{1}{2}}|\left ([\bar v_q]_{k+1}\|B_i\|_0+|D^{k}B_i|\right )\right )\\ 
&\leq \frac{C(\tilde u,\Lambda)}{\lambda_{q+1}}\left ( C_0 \delta_{q+1}^{-\sfrac{1}{2}-k}\delta_{q+1}^{\sfrac{1}{2}}\ell^{-1}+ \delta_{q+1}^{-\sfrac{1}{2}}\left ( \delta_{q+1}^{\sfrac{1}{2}}\ell^{-1}(1+\delta_{q+1}^{\sfrac{1}{2}}\ell^{-k}) + \delta_{q+1}^{\sfrac{1}{2}}\ell^{-1}\lambda_{q+1}^{k}\right )\right )\\
&\leq \frac{C(\tilde u,\Lambda)}{\lambda_{q+1}\ell}\left (C_0\delta_{q+1}^{-k}+\delta_{q+1}^{\sfrac{1}{2}}\ell^{-k}+\lambda_{q+1}^{k}\right ) \leq C(\tilde u,\Lambda) \ell^{-1}\lambda_{q+1}^{k-1}\,
\end{align*}
on $\bar B_{1-\frac{1}{2}R\delta_{q+1}}$, where we used that $C_0\delta_{q+1}^{-k} \leq \lambda_{q+1}^{k}$ for $a(b,c,C_0)$ big enough. On the other hand, on $\bar B_{1-R\delta_{q+2}}\setminus B_{1-\frac{1}{2}R\delta_{q+1}}$ we have 
\[ |D^{k}M_i| \leq  \frac{C(\tilde u,g)}{\lambda_{q+1}}\left (C(\lambda,R)C_0\delta_{q+2}^{-\sfrac{1}{2}-k}+ \delta_{q+2}^{-\sfrac{1}{2}}\lambda_{q+1}^{k}\right ) \leq C(\tilde u,g) \delta_{q+2}^{-\sfrac{1}{2}} \lambda_{q+1}^{k-1}\,,\]
where again, $a(b,c,\lambda,R,C_0)$ is chosen so large that $C(\lambda,R)C_0\delta_{q+2}^{-k}\leq \lambda_{q+1}^{k}$. 
Similarly, on $\bar B_{1-\frac{1}{2}R\delta_{q+1}}$, we find 
\begin{align*}
|D^{k}\Lambda_{ij}| &\leq C\delta_{q+1}^{\sfrac{1}{2}}\ell^{-1}\lambda_{q+1}^{k-1}\\
&\quad+C\delta_{q+1}\ell^{-2}\lambda_{q+1}^{k-2}+ \frac{C(\tilde u,\Lambda,C_0)}{\lambda_{q+1}^{2}}\left ( \delta_{q+1}^{-\sfrac{1}{2}-k}\delta_{q+1}^{\sfrac{1}{2}}\ell^{-1} + \delta_{q+1}^{-\sfrac{1}{2}}\delta_{q+1}^{\sfrac{1}{2}}\ell^{-1}\left (\lambda_{q+1}^{k}\delta_{q+1}^{-\sfrac{1}{2}} + \delta_{q+1}^{-\sfrac{1}{2}-k}\right )\right )\\ 
&\quad +\frac{C(\tilde u,\Lambda,C_0)}{\lambda_{q+1}^{2}}\left (\delta_{q+1}^{-1-k}\delta_{q+1}^{-1}+\delta_{q+1}^{-1}\delta_{q+1}^{-1-k}\right ) \\ 
&\leq C\delta_{q+1}^{\sfrac{1}{2}}\ell^{-1}\lambda_{q+1}^{k-1}+\frac{C(\tilde u,\Lambda,C_0)}{\lambda_{q+1}^{2}}\left (\delta_{q+1}^{-k}\ell^{-1}+\delta_{q+1}^{-\sfrac{1}{2}}\ell^{-1}\lambda_{q+1}^{k}\right ) + C(\tilde u,\Lambda,C_0)\delta_{q+1}^{-2-k}\lambda_{q+1}^{-2}\\ 
&\leq C(\tilde u,\Lambda)\delta_{q+1}^{\sfrac{1}{2}}\ell^{-1}\lambda_{q+1}^{k-1}\,,
\end{align*}
where we used that $\nabla \eta = 0$ in this region and that $C(C_0)\delta_{q+1}^{-1}\leq C(C_0)\ell^{-1}\leq \lambda_{q+1}$ for $a(b,c,C_0)$ large enough.
Lastly, we check the region $\bar B_{1-R\delta_{q+2}}\setminus B_{1-\frac{1}{2}R\delta_{q+1}}$:

\begin{align*}
|D^{k}\Lambda_{ij}| &\leq C(\tilde u)\lambda_{q+1}^{k-1} + \frac{C(\tilde u)}{\lambda_{q+1}^{2}}\left (\lambda_{q+1}^{k}\delta_{q+2}^{-1}+\delta_{q+2}^{-k-1}\right ) +\frac{C(\lambda,R,C_0)}{\lambda_{q+1}^{2}}\left (\delta_{q+2}^{-1-k}+\delta_{q+2}^{-\sfrac{1}{2}}\left (\delta_{q+2}^{-\sfrac{1}{2}}\lambda_{q+1}^{k}+\delta_{q+2}^{-\sfrac{1}{2}-k}\right )\right )\\ 
&\quad + C\delta_{q+2}^{-k-2}\lambda_{q+1}^{-2}+\frac{C(\lambda,R,C_0)}{\lambda_{q+1}^{2}}\left (\delta_{q+2}^{-\sfrac{1}{2}-k}\delta_{q+2}^{-\sfrac{3}{2}}+\delta_{q+2}^{-\sfrac{1}{2}}\delta_{q+2}^{-k-\sfrac{3}{2}}\right ) \\ &\quad  + \frac{C(\lambda,R,C_0)}{\lambda_{q+1}^{2}}\left (\delta_{	q+2}^{-k-1}\delta_{q+2}^{-1}+\delta_{q+2}^{-1}\delta_{q+2}^{-1-k}\right )\\
&\leq C(\tilde u) \lambda_{q+1}^{k-1} + C(\lambda,R,C_0)\delta_{q+2}^{-2-k}\lambda_{q+2}^{-2} \leq C(\tilde u)\ell^{-1}\lambda_{q+1}^{k-1}\,,
 \end{align*}
where we used $C(\lambda,R,C_0)\delta_{q+2}^{-1}\leq C(\lambda,R,C_0)\ell^{-1}\leq \lambda_{q+1}$. 
\end{proof}

Hence, if $a$ is chosen large enough, we have 
\[ \|\tau-e\|_0+\sum_{i}\|M_i\|_0+ \sum_{i,j} \|\Lambda_{ij}\|_0  < r_0\,,\]
where the norms are intended on $\bar B_{1-R\delta_{q+2}}$.  Proposition \ref{p:decomp} thus yields smooth functions $c_1,\ldots,c_{n_*}: \bar B_{1-R\delta_{q+2}}\to \R$, such that 
\begin{equation}\label{e:decomposition}
\tau = \sum_i c_i^{2}\nu_i\otimes \nu_i+ \sum_i c_i M_i + \sum_{i,j} c_ic_j\Lambda_{ij}\,,
\end{equation} 
$c_i>r_0$ on $\bar B_{1-R\delta_{q+2}}$ and for $k=0,1,2$ 
\begin{equation} \label{e:coefficients}
\|c_i\|_k \leq C(\tilde u,\Lambda)\left (1+\ell^{-k}+\ell^{-1}\lambda_{q+1}^{k-1}\right )\leq C(\tilde u,\Lambda)\left ( 1+\ell^{-1}\lambda_{q+1}^{k-1}\right )\,.
\end{equation}
\section{Proof of Proposition \ref{p:stage}: Perturbation}
Finally, we pick $\eta:= \eta_{q+1}$ from Lemma \ref{l:cutoff}, set $a_k := \eta h_q^{\sfrac{1}{2}}c_k$ and define  $v_{q+1}$ as in \eqref{d:vq+1}. Observe that, although $c_k$ is only defined in $\bar B_{1-R\delta_{q+2}}$, $a_k$ is smooth. Also, $v_{q+1}=\bar v_q= \tilde u$ on $\bar B_1\setminus B_{1-R\delta_{q+2}}$. Then, by \eqref{e:newinducedmetric} we find 
\begin{align*} 
\nabla v_{q+1}^{\intercal}\nabla v_{q+1} = \nabla  \bar v_q^{\intercal} \nabla  \bar v_q &+ \eta^{2} h_q \sum_{k=1}^{n_*}c_k^{2}\nu_k\otimes \nu_k +2\eta h_q \sum_{k=1}^{n_*}\frac{c_k}{h_q^{\sfrac{1}{2}}\lambda_{q+1}}\text{sym}\left (\nabla \bar v_q^{\intercal}B_k\right )\\
&+ 2\eta^{2}h_q\sum_{i,j=1}^{n_*}\frac{c_ic_j}{\lambda_{q+1}}\text{sym}\left (A_i^{\intercal}B_j\right )\\
& + 2\eta^{2}h_q \sum_{i,j=1}^{n_*}\frac{c_ic_j}{\lambda_{q+1}^{2}}\text{sym}\left (B_i^{\intercal}B_j\right ) + 2\eta h_q\sum_{i,j=1}^{n_*}\frac{c_ic_j}{\lambda_{q+1}^{2}}\text{sym}\left (B_i^{\intercal}C_j\nabla \eta\right ) \\
&+ 2\eta^{2}h_q \sum_{i,j=1}^{n_*}\frac{c_ic_j}{h_q^{\sfrac{1}{2}}\lambda_{q+1}} \text{sym}\left (B_i^{\intercal}C_j\nabla h_q^{\sfrac{1}{2}}\right ) +h_q\sum_{k=1}^{n_*}\frac{c_k^{2}}{\lambda_{q+1}^{2}} \nabla \eta^{\intercal}\nabla \eta \\
& + \eta^{2}h_q \sum_{k=1}^{n_*}\frac{c_k^{2}}{h_q\lambda_{q+1}^{2}}\left (\nabla h_q^{\sfrac{1}{2}}\right )^{\intercal}\nabla h_q^{\sfrac{1}{2}}\\
& +2\eta h_q \sum_{k=1}^{n_*}\frac{c_k^{2}}{h_q^{\sfrac{1}{2}}\lambda_{q+1}^{2}} \text{sym}\left (\nabla \eta^{\intercal}\nabla h_q^{\sfrac{1}{2}}\right ) +E_1\,,
\end{align*} 
where we have set 
\begin{align*}
E_1 &:= 2\eta^{2}h_q \sum_{i,j=1}^{n_*}\frac{c_i}{\lambda_{q+1}^{2}}\text{sym}\left (B_i^{\intercal}C_j\nabla  c_j\right ) +2\eta h_q^{\sfrac{1}{2}}\sum_{k=1}^{n_*}\frac{c_i}{\lambda_{q+1}^{2}}\text{sym}\left (\nabla\left (\eta h_q^{\sfrac{1}{2}}\right )^{\intercal}\nabla c_i\right )\\
&\quad+ \eta^{2}h_q\sum_{k=1}^{n_*} \frac{1}{\lambda_{q+1}^{2}}\nabla c_k^{\intercal}\nabla c_k\,.
\end{align*}
Hence we can write 
\[\nabla v_{q+1}^{\intercal}\nabla v_{q+1} = \nabla  \bar v_q^{\intercal} \nabla  \bar v_q +\eta^{2}h_q\left (\sum_{k=1}^{n_*} c_k\nu_k\otimes \nu_k  + \sum_{k=1}^{n_*} c_k M_k + \sum_{i,j=1}^{n_*} c_ic_j\Lambda_{ij} \right )+E_1+E_2\,,\] 
with  
\begin{align*} E_2:= &\,\eta(1-\eta) h_q \left (\sum_{k=1}^{n_*}c_kM_k+2\sum_{k=1}^{n_*}\frac{c_k^{2}}{h_q^{\sfrac{1}{2}}\lambda_{q+1}^{2}}\text{sym}\left (\nabla\eta^{\intercal}\nabla h_q^{\sfrac{1}{2}}\right )+2\sum_{i,j=1}^{n_*}\frac{c_i c_j}{\lambda_{q+1}^{2}}\text{sym}\left (B_i^{\intercal}C_j\nabla \eta\right )\right ) \\
 &+(1-\eta^{2})h_q\sum_{k=1}^{n_*}\frac{c_k^{2}}{\lambda_{q+1}^{2}}\nabla \eta^{\intercal}\nabla \eta\,.\end{align*}
Recalling \eqref{e:decomposition} and the definition of $\tau$ in \eqref{d:tau}, we can  see that 
\[ v_{q+1}^{\sharp}e = \bar v_q^{\sharp}e + \eta^{2}\left (\tilde g-\bar v_q^{\sharp}e -\delta_{q+2} e\right ) +E_1+E_2\,,\] 
and consequently 
\begin{align*} \tilde g-v_{q+1}^{\sharp}e &= \tilde g- \bar v_q^{\sharp }e -\eta^{2}(\tilde g-\bar v_q^{\sharp }e-\delta_{q+2}e)-E_1-E_2 = (1-\eta^{2})(\tilde g-\bar v_q^{\sharp}e)+\eta^{2}\delta_{q+2}e-E_1-E_2 \\ & = (1-\eta^{2})(\tilde g- v_q^{\sharp}e)+\eta^{2}\delta_{q+2}e-E_1-E_2\,,\end{align*}
where we used that $\bar v_q = v_q$ whenever $1-\eta^{2}>0$. We now define 
\begin{equation}\label{d:hq+1}
h_{q+1} := \frac{1-\sigma_0^{2}(1+\eta)}{1-\sigma_0^{2}(1+\eta)^{2}}(1-\eta^{2})h_q + \frac{\eta^{2}}{1-\sigma_0^{2}(1+\eta)^{2}}\delta_{q+2}\,.
\end{equation}
We have $h_{q+1}=h_q$ on $\bar B_1\setminus B_{1-R\delta_{q+2}}$ granting linearity and $|h_{q+1}'(1)|= \lambda$. Since $\sigma_0<\frac{1}{2}$ we find that on $\bar B_{1-R\delta_{q+2}}\setminus B_{1-(R+1)\delta_{q+2}}$ we have
\[ h_{q+1} \geq \frac{1}{2}(1-\eta^{2})h_q + \eta^{2}\delta_{q+2}\geq \frac{1}{2}\lambda R\delta_{q+2}(1-\eta^{2}) + \eta^{2}\delta_{q+2}= \frac{1}{2}\lambda R\delta_{q+2}+\eta^{2} \delta_{q+2}(1-\frac{1}{2}\lambda R) =:f(|x|)\,.\] 
The function $f$ is monotonically increasing since $\lambda R> 2$. Hence $h_{q+1}\geq f\geq f(0) =\delta_{q+2}$. This bound holds obviously also on $\bar B_{1-(R+1)\delta_{q+2}}$. Moreover, a rough estimate gives
\[ h_{q+1} \leq (1-\eta^{2})h_q + \frac{1}{1-4\sigma_0^{2}}\delta_{q+2} \leq (R+1)\lambda\delta_{q+2} +2\delta_{q+2}\leq 2(R+1)\lambda\delta_{q+2} \leq \Lambda \delta_{q+2}\,\]
provided $\sigma_0$ is small enough and $\Lambda(R)$ big enough, which settles \eqref{pstage:h_q1}. To show \eqref{pstage:h_q2} we define 
\[ \Phi(x) = \frac{1-\sigma_0^{2}(1+x)}{1-\sigma_0^{2}(1+x)^{2}}(1-x^{2})\,, \Psi(x)= \frac{x^{2}}{1-\sigma_0^{2}(1+x)^{2}}\,,\] 
and write 
\[ h_{q+1} = \Phi(\eta) h_q + \Psi(\eta)\delta_{q+2}\,.\]
Since $\sigma_0<\frac{1}{2}$ one finds constants $C_k$ such that 
\[ [\Phi]_k  +[\Psi]_k\leq C_k\,, k\in \N\,.\]
Then \eqref{pstage:h_q2} is a consequence of Proposition \ref{p:chain} and estimates \eqref{e:etaderivatives}. 
\section{Proof of Proposition \ref{p:stage}: Conclusion} \label{s:conclusion}
\subsection{Error estimation}
Lastly, we need to check if, once $a$ is chosen large enough, \eqref{pstage:shortness} is satisfied with $q$ replaced by $q+1$.	First of all, we show that the upper bound is true by using \eqref{pstage:shortness} to write
\begin{align*}
\tilde g-v_{q+1}^{\sharp}e &\leq (1-\eta^{2})(1+\sigma_0)h_qe + \eta^{2}\delta_{q+2} e -E_1-E_2 \\& = (1+\sigma_0(1+\eta))h_{q+1}e\\
&\qquad + \underbrace{(1-\eta^{2})(1+\sigma_0)h_qe + \eta^{2}\delta_{q+2} e -E_1-E_2 - (1+\sigma_0(1+\eta))h_{q+1}e}_{=:E}\,.
\end{align*} 
Hence, the task is to show that $E\leq 0$. First of all, on $\bar B_1\setminus B_{1-R\delta_{q+2}}$ we have $\eta \equiv 0$ and $h_{q+1} = h_q$ resulting in $E=0$. On $\bar B_{1-R\delta_{q+2}}$ we compute 
\begin{align*} 
 E &= (1-\eta^{2})(1+\sigma_0)h_qe + \eta^{2}\delta_{q+2} e -E_1-E_2\\
&\qquad\qquad -\left (\frac{1-\sigma_0^{2}(1+\eta)}{1-\sigma_0(1+\eta)}(1-\eta^{2})h_qe +\frac{\eta^{2}}{1-\sigma_0(1+\eta)} \delta_{q+2} e\right ) \\ 
 & =\left (1+\sigma_0-\frac{1-\sigma_0^{2}(1+\eta)}{1-\sigma_0(1+\eta)}\right )(1-\eta^{2})h_q e+\left (1-\frac{1}{1-\sigma_0(1+\eta)}\right )\eta^{2}\delta_{q+2} e -E_1-E_2\\
 &= \frac{-\sigma_0\eta}{1-\sigma_0(1+\eta)}(1-\eta^{2})h_q e -\frac{\sigma_0(1+\eta)\eta^{2}}{1-\sigma_0(1+\eta)}\delta_{q+2} e -E_1-E_2\,.
\end{align*}
Since $h_q\geq \lambda R\delta_{q+2}$ when $1-\eta^{2}>0$ we can conclude that 
\[  \frac{-\sigma_0(1-\eta^{2})}{2(1-\sigma_0(1+\eta))}h_q e -\frac{\sigma_0(1+\eta)\eta}{1-\sigma_0(1+\eta)}\delta_{q+2} e \leq -C(\sigma_0,\lambda,R)\delta_{q+2} e\,, \] 
for some $C(\sigma_0,\lambda,R)> 0$. Using the estimates of Lemma \ref{l:ugliestlemma} and \eqref{e:coefficients} we find the pointwise estimate
\[ |E_1|\leq C(\lambda,\Lambda,R)\frac{\delta_{q+1}}{\lambda_{q+1}^{2}\ell^{2}} \eta\,.\]
For $a$ large enough it therefore follows from \eqref{e:errorsize} that
\begin{align*} E &\leq \eta\left (C(\lambda,\Lambda, R) \frac{\delta_{q+1}}{\lambda_{q+1}^{2}\ell^{2}}e - C(\sigma_0,\lambda,R) \delta_{q+2} e\right ) -\frac{\sigma_0\eta(1-\eta^{2})}{2(1-\sigma_0(1+\eta))}h_q e -E_2 \\ 
&\leq-\frac{\sigma_0\eta(1-\eta^{2})}{2(1-\sigma_0(1+\eta))}h_q e -E_2\,.\end{align*}
To estimate this final term we recall from \eqref{e:eta'smallness} that there exists $\varepsilon>0$ such that $|\nabla \eta^{\intercal}\nabla \eta| \leq C\delta_{q+2}^{-2}\eta$ whenever $\eta\leq \varepsilon$. Consequently, when $\eta\leq \varepsilon$ we can estimate 
\[ |E_2| \leq C(\lambda,R)\eta(1-\eta)h_q\left(\ell^{-1}\lambda_{q+1}^{-1}+\delta_{q+2}^{-2}\lambda_{q+1}^{-2} \right ) \leq C(\lambda,R)\eta(1-\eta)\frac{h_q}{\ell^{2}\lambda_{q+1}^{2}} \,,\]
so that 
\[ E\leq  \eta(1-\eta)h_q\left ( \frac{C(\lambda,R)}{\ell^{2}\lambda_{q+1}^{2}}e -\frac{\sigma_0}{2}e\right ) \leq 0\,,\]
if $a(\sigma_0,\lambda,R)$ is large enough. On the other hand, when $\eta\geq \varepsilon$, then 
\begin{align*} E &\leq \eta(1-\eta)h_q\left (\frac{C(\lambda,R)}{\ell^{2}\lambda_{q+1}^{2}}e-\frac{\sigma_0}{4}e\right ) +(1-\eta^{2})h_q\left (\sum_{k=1}^{n_*}\frac{c_k^{2}}{\lambda_{q+1}^{2}}|\nabla \eta^{\intercal}\nabla\eta| e -\frac{\sigma_0\eta}{4(1-\sigma_0(1+\eta))}e\right )\\
& \leq C(1-\eta^{2})h_q\left ( \delta_{q+2}^{-2}\lambda_{q+1}^{-2}e - \frac{\sigma_0\varepsilon}{4}e\right ) \leq 0\,,
\end{align*}
if $a( \sigma_0,\varepsilon)$ is large enough. Recall in particular that $\varepsilon$ does not depend on $q$, hence we can choose $a$ depending on $\varepsilon$. This proves the upper bound in \eqref{pstage:shortness}. The lower bound is proven analoguously. 
\subsection{Estimates on $v_{q+1}$}  First of all, on $\bar B_1\setminus B_{1-R\delta_{q+2}}$ we have $v_{q+1} =\tilde u = v_q$. On the other hand, on $\bar B_{1-R\delta_{q+2}}$ we can estimate, for $k=0,1,2$,
\[ [\bar v_q - v_q]_k \leq C\ell^{2-k}[v_q]_2 + C\ell^{2-k}[\tilde u]_2 \leq \delta_{q+1}^{\sfrac{1}{2}}\ell^{1-k}\,,\] 
if $\tilde C$ in the definition \eqref{d:l} of $\ell$ is large enough. Moreover, combining the estimates of Lemma \ref{l:ugliestlemma} with estimates \eqref{e:etaderivatives}, \eqref{Dknormals} and \eqref{e:coefficients} we can estimate 
\begin{align*} [v_{q+1}-\bar v_q]_k &\leq \frac{C}{\lambda_{q+1}}\left ([\eta h_q^{\sfrac{1}{2}}c_i]_k+C(\tilde u,\Lambda)\delta_{q+1}^{\sfrac{1}{2}}\lambda_{q+1}^{k}\right ) \\ 
&\leq \frac{C(\tilde u,\Lambda)\delta_{q+1}^{\sfrac{1}{2}}}{\lambda_{q+1}}\left (\delta_{q+2}^{-k}+C_0\delta_{q+2}^{-k}+\ell^{-1}\lambda_{q+1}^{k-1}+\lambda_{q+1}^{k}\right ) \leq C(\tilde u,\Lambda) \delta_{q+1}^{\sfrac{1}{2}}\lambda_{q+1}^{k-1}  \\ 
&\leq C_0\delta_{q+1}^{\sfrac{1}{2}}\lambda_{q+1}^{k-1} \,.\end{align*}
This concludes the proof of the proposition. 

\section{Proof of Theorem \ref{t:main}}
\subsection{First approximation} Let $\sigma_0>0$ from Proposition \ref{p:stage} be given and assume that $\bar \sigma_0 <\min\{\frac{1}{2}\sigma_0,\frac{1}{4}\}$. Assume $g, u$ satisfy \eqref{tmain:ass1} and \eqref{tmain:ass2} and fix an $\alpha <\frac{1}{2}$ and a constant $x_0\in \R^{n(n+1)}$. We choose $c>b>1$ such that $\alpha <\frac{1}{2bc}$. For any $a$ big enough we now want to construct maps $v_0,h_0$ satisfying the assumptions \eqref{pstage:v_q}-\eqref{pstage:shortness} for the metric $\tilde g=g-w^{\sharp}e$, where $w\in C^{\infty}\left (\bar B_1,\R^{n(n+1)}\right )$ is a suitable map constructed in \eqref{e:w}. Then Proposition \ref{p:stage} can be applied iteratively to generate a sequence $v_q \in C^{\infty}\left (\bar B_1, \R^{m}\right )$ converging in $C^{1,\alpha}$ to a map $\underline v$ inducing the metric $\tilde g$. Setting $v=(\underline v, w)$ will then yield the wanted isometric map. First of all,  we need to do a first approximation to get into the range  of assumption \eqref{pstage:shortness}.
\begin{lemma}\label{l:firstapprox} 
Let $m\geq n+2$, $\tilde \sigma_0\in]0,\frac{1}{4}[$ and assume $u\in C^{\infty}(\bar B_1, \R^{m})$ and $h\in C^{\infty}(\bar B_1)$ satisfy \eqref{tmain:ass1}--\eqref{tmain:ass2} with $\bar \sigma_0$ replaced by $\tilde \sigma_0$. There exist $\bar\delta >0$ and $\bar \Lambda>1$ (depending only on $\tilde \sigma_0$ and $h$) such that for any positive $\delta < \bar \delta $ there exist $\tilde u\in C^{\infty}(\bar B_1, \R^{m})$, $\tilde h\in C^{\infty}(\bar B_1)$ with
\begin{align}
 (1-&\tilde \sigma_0(2+\eta))\tilde h e \leq g-\tilde u^{\sharp}e\leq (1+\tilde \sigma_0(2+\eta)) \tilde h e\,,\label{l:difference}\\ 
 &\tilde u  = u \,\text{ on } \bar B_1 \setminus B_{1-\delta} \,,\\
 &\tilde h (1) = 0 \,\text{ and } \,\tilde h  \, \text{ is linear on } \bar B_1 \setminus B_{1-\delta}\,,\label{l:h1}\\
 &\bar\Lambda^{-1} \delta \leq \tilde h \leq \bar \Lambda \delta \,\text{ on } \bar B_{1-\delta}\,, \label{l:h2}\\
 & \| D^{k}\tilde h\|_{C^{0}(\bar B_1)} \leq C\delta^{1-k} \,\text{ for } k=0,1,2,3\,,\label{l:h3}
\end{align} 
where $\eta$ is a suitable, radially symmetric, smooth cutoff function with $\eta \equiv 1 $ on $\bar B_{1-2\delta}$ and $\eta \equiv 0$ on $\bar B_1\setminus B_{1-\delta}$ and the constant $C$ in \eqref{l:h3} depends only on $|h'(1)|$. In addition, $\tilde u $ can be chosen to be arbitrarily close to $u$ in $C^{0}$.
\end{lemma}
We postpone the proof of this lemma until the end of this section and now show how to conclude the Theorem \ref{t:main} from it. Firstly, choose $\tilde \sigma_0 = \bar \sigma_0$ and fix some $\delta<\bar \delta$ to find first approximations  $\tilde u, \tilde h$ satisfying \eqref{l:difference}--\eqref{l:h3}. We then set $\lambda:= |\tilde h'(1)|$,  choose some $a> a_0(b,c,\tilde u, \sigma_0,\lambda,R,\Lambda,\delta)$ big enough to satisfy $(R+1)\delta_1<\delta$, where we recall $\delta_q = a^{-b^{q}}$. To start the iterative process we now would like to find maps $v_0,h_0$ satisfying \eqref{pstage:v_q}--\eqref{pstage:shortness}. In particular, $v_0$ will have to satisfy $\| v_0-\tilde u\|_1 <\rho_0(\tilde u)$ in order to find the normal vectorfields with the help of Proposition \ref{p:normals}. A perturbation like the one used in the proof of  Proposition \ref{p:stage} would produce a map $v_0$ satisfying most of the needed conditions, however we could only  control $\| v_0-\tilde u\|_1 \leq C\delta ^{\sfrac{1}{2}}$. Since $C\delta^{\sfrac{1}{2}}$ might be bigger than $\rho_0(\tilde u)$ such a perturbation is not sufficient. The solution, which unfortunately comes at the expense of increasing the codimension, is to perturb the metric instead: we set $v_0=\tilde u$ and find a metric $\tilde g$ of the form $\tilde g= g-w^{\sharp}e$ such that $\tilde g - \tilde u^{\sharp } e$ is very small. It is then not difficult to find $h_0$ such that $v_0,h_0$ and $\tilde g$ satisfy \eqref{pstage:w_q}--\eqref{pstage:shortness}.
To construct the map $w$ we define
\[ \tau = \frac{g-\tilde u^{\sharp}e}{\tilde h } -\frac{\delta_1}{\tilde h}e \,.\]
If $R$ is big and $\tilde \sigma_0$ is small enough we can decompose $\tau$ on $\bar B_{1-R\delta_1}$, since 
\[|\tau -e| \leq C\tilde \sigma_0 + \frac{C}{R\lambda } < r_0\,. \]
Here, we assumed that $a(\bar \Lambda)$ is taken large enough to guarantee $\bar \Lambda ^{-1}\delta \geq \lambda R\delta_1$. We can then also compute 
\[ |D^{k}\tau| \leq C(g,\tilde u)\delta_1^{-k}\,,\]
 for $k=1,2,3$. Hence, by Proposition \ref{p:decomp} we find $\nu_1,\ldots, \nu_{n*}\in \S^{n-1}$ and $c_1,\ldots, c_{n_*}\in C^{\infty}\left (\bar B_{1-R\delta_1}\right )$ with 
 \[  \tau = \sum c_i^{2}\nu_i\otimes \nu_i \,,\] 
 and, for $k=0,1,2,3$,
 \[ |D^{k}c_i| \leq C|D^{k}\tau| \leq C(g,\tilde u)\delta_1^{-k}\,\] 
 as well as the improved estimates, for $k=1,2,3$,
 \[ |\tilde h^{\sfrac{1}{2}} D^{k}c_i| \leq C(g,\tilde u)\delta_1^{\sfrac{1}{2}-k}\,.\]

 \subsection{Perturbation} 
 Fix a cutoff $\eta_0$ given by Lemma \ref{l:cutoff}, pick a constant $x_0\in \R^{n(n+1)}$ and define 
 \begin{equation}\label{e:w} w = x_0 +\sum_{k=1}^{n_*} \frac{\eta_0 \tilde h^{\sfrac{1}{2}}c_k}{\mu}\left (\sin(\mu x\cdot \nu_k)e_{k}+\cos(\mu x\cdot \nu_k)e_{n_*+k}\right )\,,\end{equation}
 where $e_i\in \R^{n(n+1)}$ is the $i-$th standard basis vector and $\mu>1$ will be chosen later.  We compute 
 \begin{align*} \nabla w  =   & \sum_{k=1}^{n_*} \eta_0\tilde h^{\sfrac{1}{2}} c_k\left ( \cos(\mu x\cdot \nu_k) e_{k} \otimes \nu_k -\sin(\mu x\cdot \nu_k) e_{n_*+k} \otimes \nu_k\right )\\
+& \frac{1}{\mu}\sum_{k=1}^{n_*} \nabla \left (\eta_0 \tilde h^{\sfrac{1}{2}}c_k\right )\left (\sin(\mu x\cdot \nu_k)e_{k}+\cos(\mu x\cdot \nu_k)e_{n_*+k}\right )\,,
 \end{align*}
 so that 
\[ \nabla w^{\intercal}\nabla w =  \eta_0^{2}\tilde h \sum_{k=1}^{n_*} c_k^{2}\nu_k\otimes \nu_k + \frac{1}{\mu ^{2}} \sum_{k=1}^{n_*} \nabla \left (\eta_0 \tilde h^{\sfrac{1}{2}}c_k\right )^{\intercal} \nabla \left (\eta_0 \tilde h^{\sfrac{1}{2}}c_k\right )\,.\]
Now we define $\tilde g = g - w^{\sharp }e$,
\[ h_0 = \frac{1-\sigma_0^{2}(2+\eta_0)}{1-\sigma_0^{2}(2+\eta_0)^{2}}(1-\eta_0^{2})\tilde h +\frac{\eta_0^{2}}{1-\sigma_0^{2}(2+\eta_0)^{2}}\delta_1\,,\]
and we claim that $\tilde g, v_0$ and $h_0$  satisfy the assumptions of Proposition \ref{p:stage}.
\subsection{Starting the process} 
First of all, since $v_0 = \tilde u$ the assumptions \eqref{pstage:v_q} are trivially satisfied once $a(\tilde u, C_0)$ is large enough. Now since $|g-\tilde u^{\sharp}e| \leq C\delta_1$ whenever $\nabla \eta_0 \neq 0$ (thanks to \eqref{l:difference}), we can estimate for $k=1,2,3$
\[ | D^{k}\left (\eta_0\tilde h^{\sfrac{1}{2}}c_k\right )| \leq C(g,\tilde u,\Lambda)\delta_1^{\sfrac{1}{2}-k}\,,\]
so that for $k=1,2$ 
\[ |D^{k}\left (w^{\sharp}e\right )| \leq C(g,\tilde u) \delta_1^{1-k} + \frac{C(g,\tilde u,\Lambda)}{\mu ^{2}} \delta_{1}^{-k-1}\leq C(g,\tilde u,\Lambda ) \delta_{1}^{1-k} \,,\]
if $\mu \geq \delta_1^{-1}$. Consequently, \eqref{pstage:w_q} is satisfied. With the same reasoning as in the proof of Proposition \ref{p:stage} we can conclude \eqref{pstage:h_q1} and \eqref{pstage:h_q2} and also  \eqref{pstage:shortness} if 
\[ \mu = \hat C \delta_1^{-1}\,\] 
for a large enough constant $\hat C$ depending on $g,\tilde u, \varepsilon$ and $\sigma_0$. Moreover, we can achieve 
\[ \| w-x_0 \|_0 < \frac{\varepsilon}{2}\,,\] 
if $\hat C $ is large enough. 
\subsection{Conclusion}
We can now apply  Proposition \ref{p:stage} iteratively to generate the sequence $v_q$. Because of the estimate \eqref{pstage:c1} the sequence converges in $C^{1}$ to a map $\underline v$ which satisfies, since we can pass to the limit in \eqref{pstage:shortness}, $\underline v^{\sharp }e = \tilde g$. Lastly, we can estimate 
\[ \|v_{q+1}-v_q\|_{1,\alpha} \leq C\|v_{q+1}-v_q\|_1^{1-\alpha}[v_{q+1}-v_q]_2^{\alpha} \leq C \delta_{q+1}^{\sfrac{1}{2}} \lambda_{q+1}^{\alpha} = C a^{-\sfrac{1}{2}b^{q}(1-2\alpha bc)}\,.\]
Since $\alpha <\frac{1}{2bc}$ the sequence converges in $C^{1,\alpha}$ and consequently $\underline v\in C^{1,\alpha}$. Setting $v=(\underline v, w)$ then concludes the proof of the main theorem. We are therefore left to proving Lemma \ref{l:firstapprox}. \\
\subsection{Proof of Lemma \ref{l:firstapprox}}
Let $r >0$ be such that  
 \begin{equation}\label{linearshortness} (1-2\tilde \sigma_0)h'(1)(|x|-1)e \leq (g- u^{\sharp }e)_x \leq (1+2\tilde \sigma_0)h'(1)(|x|-1)e\,\end{equation}
 for all $x\in \bar B_1 \setminus B_{1-r}$. Since $u$ is strictly short and $\bar B_{1-r}$ is compact we can find $\bar \rho> 0$ such that 
 \[ g-u^{\sharp }e> \bar \rho e\quad \text{on } \bar B_{1-r} \,.\]
 Fix $\rho $ such that  
 \[ 2\rho \max\{ 1, ((2\tilde \sigma_0-1)h'(1))^{-1}\} <  \min\{r,\bar \rho\}\,.\]
 With this choice we have 
 \[ g-u ^{\sharp }e \geq \rho e \,\text{ on }\bar B_{1-\delta}\,,\] 
 where we set $\delta = \rho \max\{1, ((2\tilde\sigma_0-1)h'(1))^{-1}\}$. By Lemma 1 in \cite{Laszlo}, since $ (g-u^{\sharp}e -\frac{\rho}{2}e) (\bar B_{1-\delta})$ is compact, there exist $M$ nonnegative smooth functions $a_1,\ldots,a_M\in C^{\infty}(\bar B_{1-\delta})$ and unit vectors $\nu_1, \ldots,\nu_M\in \S^{n-1}$ such that 
 \begin{equation}\label{e:l1decomp} g-u^{\sharp}e -\frac{\rho}{2}e = \sum_{i=1}^{M}a_i^{2} \nu_i\otimes \nu_i\,,\end{equation}
 on $\bar B_{1-\delta}$. Fix a radially symmetric cutoff $\eta \in C^{\infty}(\bar B_1)$ such that \begin{align} 
  \eta &\equiv 1 \, \text{ on } \bar B_{1-2\delta} \,,\\ 
  \eta & \equiv 0 \, \text{ on } \bar B_1\setminus B_{1-\delta}\,,\\
  \|\eta^{(k)}\|_0 &\leq C_k\delta^{-k} \, \text{ for } k\geq 0\,,\,\label{eta3}\\
   (\eta')^{2} &= o(\eta) \, \text{ as } \eta \to 0\,,\label{growth}
   \end{align}
Such a function can be constructed in the same way as in Lemma \ref{l:cutoff}. We now use a Nash twist to construct $\tilde u$, i.e. for $k=0,\ldots,M$ we define iteratively $u_0:= u$ and 
\[ u_k = u_{k-1} + \frac{\eta a_k }{\lambda_k}( \sin(\lambda_k x\cdot \nu_k)\zeta_k^{1}+ \cos(\lambda_k x\cdot \nu_k)\zeta_k^{2}) \,,\]
where $\lambda_k>1$ are large frequencies to be chosen and $\zeta_k^{1},\zeta_k^{2}\in C^{\infty}(\bar B_1,\R^{m})$ are orthogonal unit vector fields which are normal to $u_{k-1}$ and are provided by Lemma \ref{l: normals}. Finally we set $\tilde u := u_M$. $\tilde u$ is smooth and because of the properties of $\eta$ we certainly have $\tilde u = u$ on $\bar B_1\setminus B_{1-\delta}$. To compute the induced metric we note that 
\begin{align} \nabla u_k = \nabla u_{k-1} + \eta a_k( \cos(\lambda_k x\cdot \nu_k) \zeta_k^{1}\otimes \nu_k- \sin(\lambda_k x\cdot \nu_k)\zeta_{k}^{2}\otimes \nu_k) + O\left (\lambda_k^{-1}\right )\left (\eta+\nabla \eta\right )\end{align}
Consequently 
\begin{equation}
 \nabla u_k^{\intercal}\nabla u_k = \nabla u_{k-1}^{\intercal}\nabla u_{k-1} + \eta^{2}a_k^{2}\nu_k\otimes \nu_k +  O\left (\lambda_k^{-1}\right )\left (\eta+\nabla \eta^{\intercal}\nabla \eta\right )\,.
\end{equation}
Remembering \eqref{e:l1decomp}, we therefore find 
\[ g-\tilde u^{\sharp}e = g- u^{\sharp}e + \sum_{k=1}^{M}\left (u_{k-1}^{\sharp}e -u_k^{\sharp}e\right )= (1-\eta^{2})(g-u^{\sharp}e) + \eta^{2}\frac{\rho}{2}e-\left (\eta+\nabla \eta^{\intercal}\nabla \eta\right )\underbrace{\sum_{k=1}^{M} O\left (\lambda_k^{-1}\right )}_{=:E}\,.\]
We now set 
\[ \tilde h(x) =\frac{1-2\tilde\sigma_0^{2}(2+\eta)}{1-\tilde\sigma_0^{2}(2+\eta)^{2}}(1-\eta^{2})h'(1)(|x|-1)+ \frac{\eta^{2}}{1-\tilde\sigma_0^{2}(2+\eta)^{2}}\frac{\rho}{2}\,.\]
Then $\tilde h \in C^{\infty}(\bar B_1)$ and \eqref{l:h1} follows directly. Moreover , one can write 
\[\tilde h(x) = \Phi(\eta)h'(1)(|x|-1) + \Psi(\eta)\rho\,,\] 
for the two rational functions 
\[ \Phi(x) =\frac{1-2\tilde\sigma_0^{2}(2+x)}{1-\tilde\sigma_0^{2}(2+x)^{2}}(1-x^{2})\,,\quad \Psi(x) = \frac{x^{2}}{2-2\tilde\sigma_0^{2}(2+x)^{2}}\,.\]
Since $\tilde\sigma_0\in ]0,\frac{1}{4}[$, one easily finds a constant $C\geq 1$ such that 
\begin{equation}\label{e:rationalfunctions} [\Phi]_{C^{k}([0,1])}+[\Psi]_{C^{k}([0,1])}\leq C \,, k=0,1,2,3\,.\end{equation}
Hence,
\[ \tilde h \leq C(|h'(1)|\delta+ \rho) \leq \bar \Lambda \delta \,,\] 
everywhere and 
\[ \tilde h \geq (1-\eta^{2})h'(1)(|x|-1)+ \eta^{2}\frac{\rho}{2} \geq |h'(1)|\delta+\eta^{2}(\frac{\rho}{2}-|h'(1)|\delta) \geq \frac{\rho}{2} \geq \bar \Lambda^{-1}\delta\] 
on $\bar B_{1-\delta}$ for a suitably chosen $\bar\Lambda$ depending only on $h$ and $\tilde\sigma_0$. Hence \eqref{l:h2} is satisfied as well, while \eqref{l:h3} follows with the help of Proposition \ref{p:chain} in view of \eqref{eta3} and \eqref{e:rationalfunctions}. 
It therefore remains to show \eqref{l:difference}. On $\bar B_1\setminus B_{1-\delta} $ it is implied by \eqref{linearshortness}. If we choose $\lambda_k$ so big that $\|E\|_0 < \tilde\sigma_0\rho$, then on $\bar B_{1-2\delta}$ one finds  
\[ g-\tilde u^{\sharp} e -\tilde h e = E \leq \tilde\sigma_0 \rho e = 2\tilde\sigma_0 \tilde h e\,,\] and analoguosly 
\[ g-\tilde u^{\sharp} e -\tilde h e = E \geq -\tilde\sigma_0\rho e = -2 \tilde\sigma_0\tilde he\,.\]
We're left with the set $\bar B_{1-\delta} \setminus B_{1-2\delta} $. Observe that 
\begin{align*} (1-\tilde\sigma_0(2+\eta))\tilde h -&(1-2\tilde\sigma_0)(1-\eta^{2})h'(1)(|x|-1) \\ & = \left (\frac{1-2\tilde\sigma_0^{2}(2+\eta)}{1+\tilde\sigma_0(2+\eta)}-(1-2\tilde\sigma_0)\right )(1-\eta^{2})h'(1)(|x|-1)  +\frac{\eta^{2}}{1+\tilde\sigma_0(2+\eta)}\frac{\rho}{2} \\
&  = \frac{-\tilde\sigma_0\eta}{1+\tilde\sigma_0(2+\eta)}(1-\eta^{2})h'(1)(|x|-1)+\frac{\eta^{2}}{1+\tilde\sigma_0(2+\eta)}\frac{\rho}{2}\,,
\end{align*}
and similarly
\begin{align*} (1+\tilde\sigma_0(2+\eta))\tilde h -&(1+2\tilde\sigma_0)(1-\eta^{2})h'(1)(|x|-1) \\
&  = \frac{\tilde\sigma_0\eta}{1-\tilde\sigma_0(2+\eta)}(1-\eta^{2})h'(1)(|x|-1)+\frac{\eta^{2}}{1-\tilde\sigma_0(2+\eta)}\frac{\rho}{2}\,,
\end{align*}

 Remembering \eqref{linearshortness} we find 
 \begin{align*} g-\tilde u^{\sharp}e &\leq (1+\tilde\sigma_0(2+\eta))\tilde he  -(1+\tilde\sigma_0(2+\eta))\tilde he+(1+2\tilde\sigma_0)(1-\eta^{2})h'(1)(|x|-1)e\\
&\quad+\eta^{2}\frac{\rho}{2}e+ C(\eta +|\eta'|^{2} )|E| e \\
 &= (1+\tilde\sigma_0(2+\eta))\tilde he -\eta\left (\frac{\tilde\sigma_0(1-\eta^{2})}{1-\tilde\sigma_0(2+\eta)}h'(1)(|x|-1)e +\frac{\tilde\sigma_0(2+\eta)}{1-\tilde\sigma_0(2+\eta)}\eta\frac{\rho}{2}e\right )\\
&\quad +C(\eta +|\eta'|^{2})|E|e\,\end{align*}
 and also 
 \begin{align*}
 g-\tilde u^{\sharp}e &\geq (1-\tilde\sigma_0(2+\eta))\tilde he +\eta\left (\frac{\tilde\sigma_0(1-\eta^{2})}{1+\tilde\sigma_0(2+\eta)}h'(1)(|x|-1)e +\frac{\tilde\sigma_0(2+\eta)}{1+\tilde\sigma_0(2+\eta)}\eta\frac{\rho}{2}e\right )\\
&\qquad -C(\eta +|\eta'|^{2})|E|e\,.
 \end{align*}
 Now, because of \eqref{growth} we can find $\varepsilon$ such that 
\[|\eta'|^{2}\leq \eta \,\quad \text{ for }\eta\leq  \varepsilon\,.\]
Then, on the region where $\eta > \varepsilon$, we have 
\[ \eta\left (\frac{\tilde\sigma_0(1-\eta^{2})}{1-\tilde\sigma_0(2+\eta)}h'(1)(|x|-1)e +\frac{\tilde\sigma_0(2+\eta)}{1-\tilde\sigma_0(2+\eta)}\eta\frac{\rho}{2}e\right ) \geq C(\varepsilon)e \,,\] 
and consequently, choosing $\lambda_k$ big enough, we find 
\[ (1-\tilde\sigma_0(2+\eta))\tilde he\leq g-\tilde u^{\sharp}e \leq (1+\tilde\sigma_0(2+\eta))\tilde he\,.\]
On the other hand, when $\eta\leq \varepsilon$, it holds 
\begin{align*} g-\tilde u^{\sharp}e &\leq (1+\tilde\sigma_0(2+\eta))\tilde he -\eta\left (\frac{\tilde\sigma_0(1-\eta^{2})}{1-\tilde\sigma_0(2+\eta)}h'(1)(|x|-1)e +\frac{\tilde\sigma_0(2+\eta)}{1-\tilde\sigma_0(2+\eta)}\eta\frac{\rho}{2}e -C|E|e\right )\\
&\leq (1+\tilde\sigma_0(2+\eta))\tilde he -\eta\left (C(\varepsilon)e-C|E|e\right ) \leq (1+\tilde\sigma_0(2+\eta))\tilde he \end{align*}
if the $\lambda_k$'s are chosen large enough. The lower bound follows in the same way, concluding the proof of the lemma. 

\begin{appendix}
\section{Proofs of Propositions \ref{p:normals} and \ref{p:decomp}} 
\subsection{Proof of Proposition \ref{p:normals}}
To prove Proposition \ref{p:normals} we need the following well known lemma, an elementary proof of which is contained, for example, in \cite{CDL17}.
\begin{lemma}\label{l: normals} Let $n,d,B,u$ be as in the assumptions of Proposition \ref{p:normals}. For every $1\leq k \leq d$ there exist $\zeta_1,\ldots, \zeta_k\in C^{\infty}\left (B,\R^{n+d}\right )$ such that for all $1\leq i,j\leq d$ we have
\begin{align}
 &\langle \zeta_i, \zeta_j \rangle = \delta_{ij} \quad \text{ on } B\,,\\ & \nabla u \cdot \zeta_i = 0 \quad \,\, \text{ on } B\,.
\end{align}
\end{lemma}

\begin{proof}[Proof of Proposition \ref{p:normals}] 
In the proof all the constants appearing may depend on the embedding $u$. Fix $0< \rho_0 <1$ and let $v\in C^{\infty}(B, \R^{n+d})$ be such that $\| v-u \|< \rho_0$. Since $B$ is compact and $u$ is an embedding there exists a constant $C>0 $ such that 
\[ C^{-1}\text{Id} \leq \nabla u^{\intercal }\nabla u \leq C\,\text{Id}\,\]
in the sense of quadratic forms. Hence if $\rho_0$ is small enough we have 
\begin{equation} \label{e:p1}
 (2C)^{-1}\text{Id} \leq \nabla v^{\intercal}\nabla v \leq 2C\, \text{Id}\,,
\end{equation}
and consequently also 
\begin{equation}\label{e:p2} 
(2C)^{-n}\leq \text{det} (\nabla v^{\intercal}\nabla v ) \leq (2C)^{n}\,.
\end{equation}
Let $\zeta_1,\ldots, \zeta_m \in C^{\infty}(B, \R^{n+d})$ be the maps from Lemma \ref{l: normals} and define 
\begin{equation} \label{d:normals2}
\nu_i(v) := \zeta_i -\sum_{j=1}^{n}r_{ij}(v)\partial_j v\,,
\end{equation}
where $r_{ij}(v)$ are such that $\langle \nu_i(v),\partial_k v\rangle =0$ for every $k$. We claim that the functions $r_{ij}(v)\in C^{\infty}(B,\R^{n+d})$ depend smoothly on $\nabla v$ and satisfy the estimates 
\begin{equation}\label{e:p3}
 \|r_{ij}(v)\|_{k} \leq C_k \|v-u\|_{k+1}\quad \text{ for } k\geq 0\,.
\end{equation}
To see this, denote $b_{ik}(v) = \langle \zeta_i, \partial_k v\rangle $ and observe that 
\[ 0 = \langle \nu_i(v), \partial_k v\rangle = b_{ik}(v) - \sum_{j=1}^{n}r_{ij}(v) \langle \partial_j v, \partial_k v\rangle \,,\]
i.e. 
\[ R(v) \cdot \nabla v^{\intercal }\nabla v = B(v)\,,\]
where $R(v)$ and $B(v)$ are the $m\times n$ matrices with entries $r_{ij}(v)$ and $b_{ij}(v)$ respectively. By \eqref{e:p1}, $R(v)$ is uniquely determined. We write 
\[ (\nabla v^{\intercal } \nabla v)^{-1}_{ij} = (\det \nabla v^{\intercal }\nabla v)^{-1} P_{ij}(\nabla v)\,,\] 
where $P_{ij}(\nabla v)$ is a polynomial in the arguments $\partial_k v^{l}$. Since by assumption $[v]_1 \leq [u]_1 +1$, Lemma H\"older stuff yields  
\[ [ P_{ij}(\nabla v) ] \leq C_k [v]_{k+1}\,.\]
Moreover, \eqref{e:p2} implies 
\[ [(\det\nabla v^{\intercal }\nabla v)^{-1}]_k\leq C_k [v]_{k+1}\,,\]
so that 
\begin{equation} \label{e:p4}
 [(\nabla v^{\intercal } \nabla v)^{-1}_{ij}]_k\leq C_k [v]_{k+1}\,.
\end{equation}
For the other factor we observe that $b_{ij}(v) = \langle \zeta_i, \partial_j v- \partial_j u \rangle$,  since $\zeta_i$ is orthogonal to $Tu(B)$ at any point. Whence, by the Leibnitz rule  
\begin{equation}\label{e:p5}
[b_{ij}(v)]_k\leq C_k([v-u]_1 + [v-u]_{k+1}) \leq C_k \|v-u\|_{k+1}\,.
\end{equation}
Combining \eqref{e:p4} and \eqref{e:p5} leads to the estimate \eqref{e:p3}.\\
As a consequence, we can deduce 
\begin{equation}
\delta_{ij}-\frac{1}{2d} \leq \langle \nu_i(v),\nu_j(v)\rangle \leq \delta_{ij}+\frac{1}{2d}
\end{equation} for $\rho_0$ small enough. This implies that the family $\displaystyle \{\nu_i(v)\}_{i=1,\ldots,d}$ is linearly independent at every point and thus (being in addition orthogonal to $Tv(B)$) constitutes a frame for the normal bundle $Nv(B)$. The wanted vectorfields $\zeta_i$ are then produced by a Gram-Schmidt normalization procedure. To get the estimates \eqref{pnormals1} we carry out the procedure in details. \\
Therefore, we set 
\[ \zeta_1(v) := \frac{\nu_1(v)}{|\nu_1(v)|}\,.\]
If $\rho_0 $ is small enough, then $|\nu_i(v)| \geq \frac{1}{2} $ for every $i$ (thanks to \eqref{e:p3}), and so $\zeta_1(v)$ is a smooth function with
\[ [\zeta_1(v)]_k \leq C_k [\nu_1(v)]_k \leq C_k( 1+\|v-u\|_{k+1}) \leq C_k(1+\|v\|_{k+1})\,.\]
Moreover 
\[ |\zeta_1(v)-\zeta_1| \leq \frac{2|\nu_1(v)-\zeta_1|}{|\nu_1(v)|}\leq C\|v-u\|_1\,.\]
We now assume that $\zeta_1(v), \ldots , \zeta_{l-1}(v)$ are already constructed, satisfying \eqref{pnormals1}-\eqref{pnormals3} and in addition 
\begin{equation}\label{e:p6}
\|\zeta_i(v)-\zeta_i\|_0 \leq C\|v-u\|_1\,.
\end{equation}
 We then set 
 \[ \theta_l(v) = \nu_l(v)- \sum_{j=1}^{l-1}\langle \nu_l(v),\zeta_j(v)\rangle \zeta_j(v)\,\]
 and $ \displaystyle \zeta_l(v) = \frac{\theta_l(v)}{|\theta_l(v)|}$. It remains to show that $\zeta_l(v)$ satisfies \eqref{pnormals1}-\eqref{pnormals3} and \eqref{e:p6}. \\ 
 Observe that 
 \[ \langle \nu_l(v), \zeta_j(v)\rangle  = \langle \nu_l(v)-\zeta_l,\zeta_j(v)\rangle+ \langle \zeta_l,\zeta_j(v)-\zeta_j\rangle \,\]
 so that $\|\langle \nu_l(v),\zeta_j(v)\rangle \|_0 \leq C\|v-u\|_1$ and  
 \begin{align*}
  [ \langle \nu_l(v), \zeta_j(v)\rangle]_k &\leq C_k(1+[r_{ij}(v)]_k+\|r_{ij}(v)\|_0[v]_{k+1}+\|v-u\|_1(1+\|v\|_{k+1}) +[\zeta_j(v)-\zeta_j]_k)\\ &\leq C_k(1+\|v\|_{k+1})\,.
 \end{align*}
 In particular $|\theta_l(v)|\geq \frac{1}{4}$ for $\rho_0$ small enough and 
 \[ [\theta_l(v)]_k \leq C_k (1+\|v\|_{k+1})\,.\]
 Therefore $\zeta_l(v)$ satisfies \eqref{pnormals1}-\eqref{pnormals3}. Since moreover 
 \begin{align*}|\zeta_l(v)- \zeta_l |&\leq \frac{2|\theta_l(v)-\zeta_l|}{|\theta_l(v)|} \leq C( |\theta_l(v)-\nu_l(v)|+|\nu_l(v)-\zeta_l|) 
 \\&\leq C\|v-u\|_1\end{align*}
the proposition is proved.
\end{proof} 
\subsection{Proof of Proposition \ref{p:decomp}}
For the proof of Proposition \ref{p:decomp} we need the following lemma from \cite{CDS10}. 

\begin{lemma}\label{l:decomp}
Let $g_0\in \text{Sym}_{n}^{+}$. There exists $r \equiv r(g_0,n)>0$, $\nu_1,\ldots, \nu_{n_*}\in \mathbb{S}^{n-1}$, and linear maps $L_1,\ldots,L_{n_*}: \text{Sym}_n \to \R$ such that 
\[ g = \sum_{k=1}^{n_*} L_k(g)\nu_k\otimes \nu_k\,,\] 
for every $g\in \text{Sym}_n$. Moreover, if $g\in \text{Sym}_n$ is such that $|g-g_0|<r$, then $L_k(g) > r$ for every $k$. 
\end{lemma}
Now the proposition is an easy consequence of the classical implicit function theorem. 
\begin{proof}[Proof of Proposition \ref{p:decomp}]
Let $r>0$ be the radius and $\nu_1,\ldots, \nu_{n_*}\in \mathbb{S}^{n-1}$  be the vectors given by Lemma \ref{l:decomp} when $g_0 = \text{Id}_n$ and define  the map
\begin{align*}\Psi : &\left (\text{Sym}_n \right )^{n_*^2}\times \left( \text{Sym}_n \right )^{ n_*}\times \R^{n_*}\times \R^{n_*}\to \text{Sym}_n\\
&\left (\{\Lambda_{ij}\},\{M_i\},g,\{c_i\}\right )\mapsto \sum_{i}^{n*}c_i^{2}\nu_i\otimes \nu_i + \sum_{i=1}^{n_*}c_i M_i +\sum_{i,j=1}^{n_*}c_ic_j\Lambda_{ij} -g\,.
\end{align*}
$\Psi$ is smooth and by Lemma \ref{l:decomp} there exist $\bar c_1,\ldots,\bar c_{n*}\in \R$ with $\bar c_j > r$ for every $j$ and 
\[ \Psi(0,0,\text{Id}_n,\{\bar c_j\}) = 0\,, \quad \partial_{c_i} \Psi \vert_{(0,0,\text{Id}_n, \{\bar c_j\})} =  2 \bar c_i \nu_i\otimes \nu_i \,.\]
Since the family $\{\nu_i\otimes \nu_i\}$ is linearly independent the differential of $\Psi$ with respect to the variable $c=(c_1,\ldots,c_{n*})$ has full rank at $(0,0,\text{Id}_n,\bar c)$. Consequently, by the implicit function theorem, there exist neighborhoods $V$ of $(0,0,\text{Id}_n)$ and $U$ of $\bar c$ respectively and a diffeomorphism $\Phi:V\to U$ such that 
\[ \{ \Psi = 0 \} \cap \left ( V\times \R^{n_*}\right ) = \{ (\{\Lambda_{ij}\},\{M_i\}, g, \Phi(\{ \Lambda_{ij}\},\{M_i\},g) ) : (\{ \Lambda_{ij}\},\{M_i\},g)\in V\}\,.\]
Therefore, if $r_0$ is small enough we can define $c_k(x) := \Phi(\{\Lambda_{ij}(x)\},\{M_i(x)\}, \tau(x))_k$ and \eqref{e:pdecomp1} will be satisfied. The estimates \eqref{e:pdecomp2} are then a consequence of Proposition \ref{p:chain}.
\end{proof}
\end{appendix}

\bibliographystyle{plain}

\end{document}